\documentclass[11pt]{article} % use larger type; default would be 10pt

\usepackage[utf8]{inputenc} % set input encoding (not needed with XeLaTeX)

\usepackage{geometry} % to change the page dimensions
\usepackage{graphicx} % support the \includegraphics command and options

\usepackage{array} % for better arrays (eg matrices) in maths
\usepackage{paralist} % very flexible & customisable lists (eg. enumerate/itemize, etc.)
\usepackage{verbatim} % adds environment for commenting out blocks of text & for better verbatim
\usepackage{subfig} % make it possible to include more than one captioned figure/table in a single float
\usepackage{tikz}
\usetikzlibrary{positioning,decorations.text}

\usepackage{fancyhdr} % This should be set AFTER setting up the page geometry
\usepackage{sectsty}
\allsectionsfont{\sffamily\mdseries\upshape} % (See the fntguide.pdf for font help)
\usepackage[nottoc,notlof,notlot]{tocbibind} % Put the bibliography in the ToC
\usepackage[titles,subfigure]{tocloft} % Alter the style of the Table of Contents

\usepackage{mathtools}%
\mathtoolsset{showonlyrefs}%
\usepackage{xcolor}

\usepackage{amsfonts}
\usepackage{amsmath}
\usepackage{amssymb}
\usepackage{graphicx}
\usepackage{amsthm}
\usepackage{dsfont}    % for making the fat 1 like in characteristic function
\usepackage{microtype} % Apperaently increases the asthetics of the document
\usepackage{amsthm}
\usepackage{cite}

\usepackage{filecontents}
\newtheorem{thm}{Theorem}
\newtheorem{lem}[thm]{Lemma}
\newtheorem{prop}[thm]{Proposition}
\newtheorem{cor}[thm]{Corollary}

\newtheorem{conj}[thm]{Conjecture}
\theoremstyle{remark}

\newcommand{\R}{\mathbb{R}}

\newcommand{\CJ}{\mathcal{J}}

\newcommand{\EF}{\mathrm{E_{log}}}
\renewcommand{\S}{\mathbb{S}}

\newcommand{\Impl}{\mathop{\qquad\Rightarrow\qquad}}

\title{On Gegenbauer Point Processes on the unit interval
\footnote{Data sharing not applicable to this article as no datasets were generated or analysed during the current study.}}
\author{Carlos Beltrán\thanks{CB ( ORCID: 0000-0002-0689-8232; \textbf{corresponding author}): Universidad de Cantabria. Partially supported by Ministerio de
Economía y Competitividad (MINECO) through grants MTM2017-83816-P,
MTM2017-90682-REDT and by Banco de Santander and Universidad de Cantabria
grant 21.SI01.64658.} , Antonia Delgado$^\dagger$, Lidia Fernández\thanks{AD (ORCID: 0000-0003-1547-6653) and LF (ORCID: 0000-0001-7418-3231): Universidad de Granada. Partially supported by Ministerio de Ciencia, Innovación y Universidades (MICINN), and Fondo Europeo de Desarrollo Regional (FEDER), through grant PGC2018-094932-B-I00.} , Joaquín Sánchez--Lara\thanks{JSL (ORCID: 0000-0003-1969-9643): Universidad de Granada. Partially supported by Ministerio de Economía y Competitividad (MINECO), through grant MTM2015-71352-P.}}

\date{\today} % Activate to display a given date or no date (if empty),
\begin{document}

\maketitle
\begin{abstract}
  In this note we compute the logarithmic energy of points in the unit interval [-1,1] chosen from different Gegenbauer Determinantal Point Processes. We check that all the different families of Gegenbauer polynomials yield the same asymptotic result to third order, we compute exactly the value for Chebyshev polynomials and we give a closed expresion for the minimal possible logarithmic energy. The comparison suggests that DPPs cannot match the value of the minimum beyond the third asymptotic term.
\end{abstract}

\section{Introduction}
 Let $\mathrm{K}\subseteq\R^d$ be an infinite compact set. The logarithmic energy of $n$ points $ x_1,\ldots,x_n\in \mathrm{K}, x_i\neq x_j$ for $i\neq j$ is
\begin{equation}\label{eq:Elog}
\EF(x_1,\ldots,x_n)=-\sum_{i\neq j}\log\|x_i-x_j\|.
\end{equation}
The minimum value of this energy is then
\[
\mathcal{E}_{\mathrm log}(\mathrm{K},n)=\min_{x_1,\ldots,x_n\in \mathrm{K}}\EF(x_1,\ldots,x_n).
\]
Points which minimize the logarithmic energy (i.e. points with energy equal to $\mathcal{E}_{\mathrm log}(\mathrm{K},n)$) are called Fekete points of $\mathrm{K}$. Note that an alternative formulation is: a collection of $n$ points in $\mathrm K$ is a set of Fekete points if the product of their mutual distances is as large as it can be.

The study of Fekete points is an area of very active research, see the classical surveys \cite{W1952,KS1997} or the recent and very complete monography \cite{BHS2019}. The case that $\mathrm{K}=\S^2\subseteq\R^3$ is the unit sphere is the topic of problem number $7$ in Smale's list \cite{S1998} (the problem was actually posed by Shub and Smale in \cite{SS1993}). In particular, it is of greatest interest to describe the value of $\mathcal{E}_{\mathrm log}(\mathrm{K},n)$.

This problem of finding Fekete points is fully solved if $\mathrm{K}=[-1,1]$ is the unit interval since Fejer \cite{F1932}: the unique minimizer of \eqref{eq:Elog} is the set consisting on the $2$ extremes of the interval, plus the $n-2$ zeros of a Gegenbauer polynomial, namely $C_{n-2}^{3/2}$. Recall that for any fixed $\lambda>0$, the sequence of Gegenbauer polynomials $C_n^\lambda $ is a sequence of polynomials of degree $n=0,1,2,\ldots$ such that
\[
C_n^\lambda (1)=\binom{2\lambda+n-1}{n},\quad \int_{-1}^1C_n^\lambda (x)C_m^\lambda (x)w^\lambda(x)\,dx=0\quad \text{for $m\neq n$},
\]
where
\begin{equation}\label{eq:w}
 w^\lambda (x)=\frac{\Gamma(\lambda+1)}{\sqrt{\pi}\Gamma(\lambda+1/2)}(1-x^2)^{\lambda-1/2}.
\end{equation}
That is to say, they are a sequence of orthogonal polynomials w.r.t. the weight $w^\lambda(x)$.  We have not found in the literature an explicit value of the minimal logarithmic energy in the unit interval (which is a nontrivial calculation from the description of Fekete points). This is our first result:
\begin{thm}\label{th:EnergiaMinimaIntervalo}
Let $\epsilon_n=\mathcal{E}_{\mathrm log}([-1,1],n)$ be  the energy of $n$ Fekete points in the interval $[-1,1]$. Then,
\begin{align*}
\epsilon_n=&-n(n-1)\log (2)-3(n-1)\log(n-1)-n\log(n)-4\sum_{j=1}^{n-2}j\log(j)+ \sum_{j=1}^{2n-2}j\log (j)\\
=&\log(2) n^2-n\log(n)-2\log(2)n-\frac{1}{4}\log(n)+O(1).
\end{align*}
\end{thm}
Zeros of sequences of orthogonal polynomials are known to be well--distributed according to different criteria, and more generally in other contexts zeros of functions of increasing degree exhibit good separation properties and attain somehow low values of the logarithmic energy. For example, the eigenvalues of $n\times n$ matrices of the form $A^{-1}B$ where $A,B$ are random with complex Gaussian entries (called the {\em spherical ensemble} \cite{K2006,AZ2015}) taken onto the unit sphere $\S^2\subseteq\R^3$ through the inverse stereographic projection have, on the average, a logarithmic energy that matches the minimum value $\mathcal{E}_{\mathrm log}(\S^2,n)$ up to second order asymptotics. The zeros of random polynomials whose coefficients are complex Gaussians with some carefully chosen variances, sent again to the sphere with the same projection, have even better properties regarding logarithmic energy, see \cite{ABS2011}.

However in general for other sets it is not possible to find families of functions whose zeros have these good properties, and it is thus quite common to take alternative approaches. A very popular method is that of Determinantal Point Processes (DPPs) that we briefly describe now (see \cite{HKPV2009} for a detailed description of the theory).

 Let $n\geq1$. We endow the infinite compact set $\mathrm K$  with some measure $\sigma$, and we let $\mathcal{H}\subseteq L^2(\mathrm{K},\sigma)$ be some $(n+1)$--dimensional linear subspace of the space of squared integrable functions defined on $\mathrm{K}$. Following \cite{HKPV2009}, the Macci-Soshnikov theorem implies that there exists a random point process associated with $\mathcal{H}$ that has the following properties:
 \begin{enumerate}
 \item With probability $1$, the random process outputs $n+1$ different points in $\mathrm{K}$.
 \item For any given $\sigma$--integrable function $f:\mathrm{K}\times \mathrm{K}\to\R$ (possibly undefined in the diagonal), the expected value of the  sum $\sum_{i\neq j} f(x_i,x_j)$ when $x_0,\ldots,x_n$ are the output of the random process, is given by the formula
 \begin{equation}\label{eq:dpps}
 \mathrm E\left(\sum_{i\neq j} f(x_i,x_j)\right)=\iint_{x,y\in \mathrm K}(K_\mathcal{H}(x,x)K_\mathcal{H}(y,y)- |K_\mathcal{H}(x,y)|^2)f(x,y)\,d\sigma(x)\,d\sigma(y),
 \end{equation}
 where $K_\mathcal{H}:\mathrm K\times\mathrm K\to\R$ is the projection kernel onto $\mathcal{H}$. In other words, if $f_0,\ldots,f_n$ is an orthogonal basis of $H$ then
 \[
 K_\mathcal{H}(x,y)=\sum_{j=0}^n f_j(x)f_j(y),\quad x,y\in \mathrm K.
 \]
 \end{enumerate}
The second property implies that the random points exhibit some repulsion, and has been used to give upper bounds on the minimum value of the energy $\mathcal{E}_{\mathrm log}(\mathrm{K},n)$ (and other energies) for different sets: in \cite{AZ2015} for the $2$--sphere $\S^2$, in \cite{BMOC2016} for the $d$--sphere  $\S^d$, in \cite{BE2018} (see also \cite{B2014}) for the complex projective space, in \cite{MOC2018} for the flat torus and in \cite{BF---} for the rotation group $SO(3)$.

Despite all this success, an important gap in the theory remains open: how do DPPs compare to optimal point distribution in the case of the unit interval? Namely, although in that case we know exactly the position of the Fekete points and the value of the minimal energy (Theorem \ref{th:EnergiaMinimaIntervalo}), it seems that no deep study has been made of the value of the energy of points coming from the DPPs in that interval (see \cite{HM2019} for some computations that would lead to first order asymptotics), and  a comparison of the power of this technique with the exact solution is in order.

In this paper we study the following questions:
\begin{enumerate}
\item[I)] Given a measure $\sigma$ in $[-1,1]$, the most straightforward choice for the $(n+1)$--dimensional subspace $\mathcal{H}\subseteq L^2([-1,1],\sigma)$ is the subspace generated by the first $n+1$ orthogonal polynomials associated to $\sigma$. Which is the optimal choice of $\sigma$? Namely, which choice of $\sigma$ gives the smallest value for \eqref{eq:dpps}?
\item[II)] Fekete points are actually the $2$ extreme points, plus the zeros of a Gegenbauer polynomial. Is it better also to take the extreme points and then the points coming from a DPP, or is it better to just take the points coming from the DPP?
\item[III)] A collection of natural choices for $\sigma$ is given by $d\sigma=w^\lambda dx$ where $\lambda\in(-1/2,\infty)$, $w^\lambda$ is given by \eqref{eq:w}, and $dx$ is Lebesgue measure in $[-1,1]$. The associated subspace for $n+1$ points is then generated by the first $n+1$ Gegenbauer polynomials $C_{k}^\lambda $, $0\leq k\leq n$. How does the expected energy of points coming from these processes, computed via \eqref{eq:dpps}, compare to the minimal energy $\epsilon_{n+1}$ given in Theorem \ref{th:EnergiaMinimaIntervalo}?
\end{enumerate}
We now state our second main result which solves item III) and its corollary that gives a partial answer to item II).

\begin{thm}\label{th:main2}
Let $d\sigma=w^\lambda dx$, $\lambda\in(-1/2,\infty)$ fixed, and let $\mathcal{H}_{n+1}$ be the subspace generated by the first $n+1$ orthogonal polynomials associated to $\sigma$. In other words,
\[
\mathcal{H}_{n+1}=Span(C_{0}^\lambda ,\ldots,C_{n}^\lambda ).
\]
Then, the expected value of the logarithmic energy of the corresponding DPP is
\[
\mathrm E(\lambda,n+1)=(n+1)^2\log2-(n+1)\log (n+1)+(1-\gamma-2\log 2)n+o(n).
\]
In particular, the dependence on $\lambda$ falls into the remainder $o(n)$. Moreover, in the case $\lambda=0$ (which corresponds to Chebyshev polynomials) we get the exact expresion
\begin{align*}
\mathrm E(0,n+1)=&(n+1)^2\log 2-\left((n+1)\log2+\frac34H_n+nH_{2n-1}+\frac{H_{2n}}{2}-n+\frac12\right)\\
=&(n+1)^2\log2-(n+1)\log (n+1)+(1-\gamma-2\log 2)n-\frac14\log n+O(1),
\end{align*}
where $H_n=1+1/2+\cdots+1/n$ is the $n$--th harmonic number and $\gamma$ is the Euler--Mascheroni constant.
\end{thm}
Comparing the asymptotics of Theorem  \ref{th:main2} with $\epsilon_{n+1}$ from Theorem \ref{th:EnergiaMinimaIntervalo} we conclude that the points generated by the DPP have a greater constant in the $O(n)$ term: $1-\gamma-2\log2$ versus $-2\log2$. Thus, there is an excess of $(1-\gamma)n\approx 0.423n$. It thus seems that the power of the technique of DPPs is unsuficient to fit the correct $O(n)$ term in the minimal energy expression.

Our partial answer to item II) above is:
\begin{cor}\label{cor:puroymixto}
Let $\lambda\in(-1/2,\infty)$ and consider the two following point processes in the unit interval:
\begin{enumerate}
\item Generate $n+3$ points using the point process of Theorem \ref{th:main2}.
\item Generate $n+1$ points using the point process of Theorem \ref{th:main2}, and add the two extremes $\pm1$ of the unit interval.
\end{enumerate}
Then, the expected energy of the two point processes is equal up to $o(n)$. Moreover, if $\lambda=0$ then the first process has smaller energy than the second one, the difference being in the $O(\log n)$ term.
\end{cor}
We finish the introduction suggesting a solution to the optimal measure problem:
\begin{conj}
The answer to Question I) above is: for any fixed $n\geq1$, the optimal measure is $d\sigma=w^\lambda dx$ for some $\lambda\in(-1/2,\infty)$. The value of $\lambda$ may depend on $n$.
\end{conj}
\section{Proof of Theorem  \ref{th:EnergiaMinimaIntervalo}}\label{sec:proof1}

%El resultado de esta sección es el siguiente teorema
%\begin{thm}
%La energía mínima para $n$ cargas libres en el intervalo $[-1,1]$ es
%$$\epsilon_n=-n(n-1)\log 2-3(n-1)\log(n-1)-n\log(n)-4\sum_{j=1}^{n-2}j\log(j)+\sum_{j=1}^{2n-2}j\log j$$
%y su asintótica para $n\to\infty$
%$$\epsilon_n=\log(2) n^2-n\log(n)-2\log(2)n-\frac{1}{4}\log(n)+O(1)\,.$$
%\end{thm}
%\begin{remark}A partir de la expresión explícita para $\epsilon_n$, en su día sacamos con Mathematica que
%$$\epsilon_n=\log(2)\, n^2 - n \log n - \log(4) \, n - \frac{1}{4} \log n
%+ \frac{13 \log 2}{12} - 3 \log A + \frac{1}{4 n} + \frac{23}{192 n^2} + \frac{7}{96 n^3} + \frac{61}{1280 n^4} + \frac{31}{960 n^5} + \frac{991}{43008 n^6}+\dots$$
%donde $A\approx 1.28243$ es la constante de Glaisher que es la solución de
%$$\log A =\frac{1}{12}-\zeta'(-1)\,,$$
%y $\zeta$ es la función zeta de Riemann.
%\end{remark}
%
%El resto de la sección está dedicada a demostrar este teorema. TIENE PINTA DE QUE LAS CUENTAS SE PUEDEN DEPURAR.

%Para esta sección  necesitaremos
%$$C_n^\lambda(1)=\binom{n+2\lambda-1}{n}=\frac{\Gamma(n+2\lambda)\Gamma(2\lambda)}{(n-1)!}\Impl
%C_{n-2}^{3/2}(1)=\frac{\Gamma(n-2+3)\Gamma(3)}{(n-1)!}=\frac{n(n+1)}{2}$$
%y que el líder de $C_n^\lambda$ es
%$$C_n^\lambda(x)= \frac{2^n(\lambda)_n}{n!}x^n+\text{term. menor grado}\,.$$
%Denotaremos por $\kappa$ al líder de $C_{n-2}^{3/2}$
%$$\kappa=\frac{2^{n-2}(3/2)_{n-2}}{(n-2)!}\,.$$
Recall Jacobi's polynomials
$$P_n^{(\alpha,\beta)}(x)= \frac{1}{2^n}\binom{2n+\alpha+\beta}{n}x^n+O(x^{n-1})$$
that are a generalization of Gegenbauer's polynomials, in particular satisfying $C_{n-2}^{3/2}=P_{n-2}^{(1,1)}$, see for example  \cite{AS1964} where we also find the following classical facts:
$$P_n^{(\alpha,\beta)}(1)=\binom{n+\alpha}{n}=\frac{(\alpha+1)_n}{n!}\Impl
P_{n-2}^{(1,1)}(1)=\frac{(2)_{n-2}}{(n-2)!}=(n-1).$$

Let  $\kappa$ be the leading coefficient of $P_{n-2}^{(1,1)}$,
$$\kappa=\frac{1}{2^{n-2}}\binom{2n-2}{n-2}
%=\frac{1}{2^{n-2}}\frac{(2n-2)!}{n!(n-2)!}
=\frac{1}{2^{n-2}}\frac{(n+1)_{n-2}}{(n-2)!}\,.$$
We need to analyze the energy of the $n$ point set consisting of $x_1=-1$, $x_n=1$ and $x_2,\dots, x_{n-1}$ the zeros of $P_{n-2}^{(1,1)}$, that is
\begin{align*}
\epsilon_n&=\sum_{i\neq j}\log\frac{1}{|x_i-x_j|}=\\
&=-2\log 2 -2\log\prod_{i=2}^{n-1}(1-x_i)-2\log\prod_{i=2}^{n-1}(1+x_i)-2\log\prod_{1<i<j<n}|x_i-x_j|\\
&=-2\log 2-2\log(P_{n-2}^{(1,1)}(1))-2\log((-1)^{n-2}P_{n-2}^{(1,1)}(-1))+4\log\kappa
-2\log\prod_{1<i<j<n}|x_i-x_j|\\
&=-2\log 2-4\log(P_{n-2}^{(1,1)}(1))+4\log\kappa
-2\log\prod_{1<i<j<n}|x_i-x_j|\\
&=-2\log 2-4\log(n-1)+4\log\kappa-2\log\prod_{1<i<j<n}|x_i-x_j|\\
&=-2\log 2 -4\log(n-1) +(4+2(n-2)-2)\log\kappa- \log\Big(\kappa^{2(n-2)-2}\prod_{1<i<j<n}(x_i-x_j)^2\Big)
\\
&=-2\log 2 -4\log(n-1) +2(n-1)\log\kappa-\log D_{n-2}^{(1,1)},
\end{align*}
where
%$$D_n^{(\alpha,\beta)}=2^{-n(n-1)}\prod_{j=1}^nj^{j-2n+2}(j+\alpha)^{j-1}(j+\beta)^{j-1}
%(n+\alpha+\beta+j)^{n-j}\,,$$
%por lo que en nuestro caso
$$D_{n-2}^{(1,1)}=2^{-(n-2)(n-3)}\prod_{j=1}^{n-2}j^{j-2n+6}(j+1)^{2j-2}
(n+j)^{n-2-j}\,$$
is the discriminant of $P_{n-2}^{(1,1)}$, whose value is known (see  \cite[Th. 6.71, p. 143]{S1975}). Plugging this and the value of $\kappa$ in the formula for $\epsilon_n$ we thus conclude:
\begin{align*}
\epsilon_n
%&=-2\log 2 -4\log(n-1) +2(n-1)\log\kappa\\
%&\quad-\log\left(2^{-(n-2)(n-3)}\prod_{j=1}^{n-2}j^{j-2n+6}(j+1)^{j-1}(j+1)^{j-1}
%(n+j)^{n-2-j}\right)\\
%&=-2\log 2 -4\log(n-1) +2(n-1)\log\frac{(n+1)_{n-2}}{2^{n-2}(n-2)!}\\
%&\quad+(n-2)(n-3)\log 2 -\sum_{j=1}^{n-2}\bigg((j-2n+6)\log(j)+2(j-1)\log(j+1)+(n-2-j)\log(n+j)\bigg)\\
&=-n(n-1)\log 2-4\log(n-1)+2(n-1)\left(\sum_{j=1}^{n-2}\log(n+j)- \sum_{j=1}^{n-2}\log(j)\right)\\
&\quad- \sum_{j=1}^{n-2}\Big((j-2n+6)\log(j)+2(j-1)\log(j+1)+(n-2-j)\log(n+j)\Big).
\end{align*}
Expanding and simplifying this expresion we get the simpler formula claimed in the theorem.
%\[
%-n(n-1)\log 2-3(n-1)\log(n-1)-n\log(n)-4\sum_{j=1}^{n-2}j\log(j)+\sum_{j=1}^{2n-2}j\log j.
%\]
%continuando
%\begin{align*}
%\epsilon_n&=-n(n-1)\log 2-4\log(n-1)-\sum_{j=1}^{n-2}\bigg((j+4)\log(j)+2(j-1)\log(j+1)-(n+j)\log(n+j)\bigg)\\
%&=
%-n(n-1)\log 2-4\log(n-1)-\sum_{j=1}^{n-2}\bigg((j+4)\log(j)+2(j-2)\log(j)-(n+j)\log(n+j)\bigg)\\
%&\quad-2(n-3)\log(n-1)+2(1-2)\log(1)\\
%&=-n(n-1)\log 2-2(n-1)\log(n-1)-\sum_{j=1}^{n-2}\bigg(3j\log(j)-(n+j)\log(n+j)\bigg)\\
%&=-n(n-1)\log 2-2(n-1)\log(n-1)-\sum_{j=1}^{n-2}3j\log(j)+\sum_{j=n+1}^{2n-2}j\log(j)\\
%&=-n(n-1)\log 2-2(n-1)\log(n-1)
%-\sum_{j=1}^{n-2}4j\log(j)+\sum_{j=1}^{2n-2}j\log j-n\log n-(n-1)\log(n-1)\\
%&=-n(n-1)\log 2-3(n-1)\log(n-1)-n\log(n)-4\sum_{j=1}^{n-2}j\log(j)+\sum_{j=1}^{2n-2}j\log j\,,
%\end{align*}
%lo que demuestra la expresión para $\epsilon_n$ dada en el teorema. Para el resultado referente a la asintótica necesitamos el siguiente lema.
Using
$$\log(n-a)=\log(n)+\log\left(\frac{n-a}{n}\right)=\log(n)+\log\left(1-\frac{a}{n}\right)
=\log(n)-\frac{a}{n}+O(n^{-2})
$$
and Lemma \ref{LemaHiperFactorial} we also get the value for the asymptotic expansion of $\epsilon_n$, finishing the proof of the theorem.
%\begin{align*}
%\epsilon_n&=
%-n(n-1)\log 2-3(n-1)\log(n-1)-n\log(n)\\
%&\quad-4\left(\frac{1}{2}(n-2)^2\log(n-2)-\frac{1}{4}(n-2)^2+\frac{1}{2}(n-2)\log(n-2)+\frac{1}{12}\log(n-2)+O(1)\right)\\
%&\quad+\left(\frac{1}{2}(2n-2)^2\log(2n-2)-\frac{1}{4}(2n-2)^2+\frac{1}{2}(2n-2)\log(2n-2)
%+\frac{1}{12}\log(2n-2)+O(1)\right)\\
%&=-n^2\log 2-3n\log(n-1)-n\log(n)+n\log(2)+3\log(n-1)\\
%&\quad-2n^2\log(n-2)+n^2+6n\log(n-2)-4n-\frac{13}{3}\log(n-2)+O(1)\\
%&\quad+2n^2\log(2n-2)-n^2-3n\log(2n-2)+2n+\frac{13}{12}\log(2n-2)+O(1)\,.
%%&=\log(2)\, n^2 - n \log n - \log(4) \, n - \frac{1}{4} \log n+O(1)\,,
%\end{align*}
%Using
%$$\log(n-a)=\log(n)+\log\left(\frac{n-a}{n}\right)=\log(n)+\log\left(1-\frac{a}{n}\right)
%=\log(n)-\frac{a}{n}+O(n^{-2})
%$$
%we then have
%\begin{align*}
%\epsilon_n&=-n^2\log 2-3n\log(n)-n\log(n)+n\log(2)+3\log(n)+O(1)\\
%&\quad-2n^2\left(\log(n)-\frac{2}{n}+O(n^{-2})\right)+n^2+6n\left(\log(n)+O(n^{-1})\right)
%-4n-\frac{13}{3}\log(n)+O(1)\\
%&\quad+2n^2\left(\log(n)+\log(2)-\frac{1}{n}+O(n^{-2})\right)-n^2
%-3n\left(\log(n)+\log(2)+O(n^{-1})\right)+2n+\frac{13}{12}\log(n)+O(1)\\
%&=-\log(2)n^2-4n\log(n)+\log(2)n+3\log(n)+O(1)\\
%&\quad-2n^2\log(n)+n^2+6n\log(n)+(4-4)n-\frac{13}{3}\log(n)+O(1)\\
%&\quad+2n^2\log(n)+(2\log(2)-1)n^2-3n\log(n)+(-2-3\log(2)+2)n+\frac{13}{12}\log(n)+O(1)\\
%&=\log(2) n^2-n\log(n)-2\log(2)n-\frac{1}{4}\log(n)+O(1)\,,
%\end{align*}
%lo que demuestra el resultado sobre la asintótica de la energía mínima.

\section{First integral formulas}\label{sec:firstformulas}
In order to compute the expected value of the logarithmic energy for points coming from the DPPs, we need to write down the integral in \eqref{eq:dpps} for our choice of $\sigma$. In order to do so, we first need an orthonormal sequence of polynomials. Fix $\lambda\in(-1/2,\infty)$, $\lambda\neq0$, and consider the normalized Gegenbauer polynomials
$$\widehat{C}_n^\lambda=\gamma_n^\lambda C_n^\lambda\,,\qquad \text{with }\quad
\gamma_n^\lambda=%\frac{1}{\sqrt{h_n}}=
\sqrt{\frac{n!(n+\lambda)}{\lambda (2\lambda)_n}}\,,$$
which describe an orthonormal basis of the subspace they span.

Recall the following identities
\begin{align*}
\frac{d}{dx}C_n^\lambda(x)=&2\lambda C_{n-1}^{\lambda+1}(x),\\
\frac{d}{dx}\left[C_{n-1}^{\lambda+1}(x)w^{\lambda+1}(x)\right]=&
-\frac{n(n+2\lambda)(\lambda+1)}{\lambda(2\lambda+1)}C_n^\lambda(x)w^\lambda(x).
\end{align*}
We can now compute the kernel $K_n^\lambda$ (for the point process of $n+1$ points) using the  Christoffel--Darboux summation formula (see \cite{S1975}):
\begin{align}
K_n^\lambda(x,y)&=\sum_{j=0}^n\widehat{C}_j^\lambda(x)\widehat{C}_j^\lambda(y)=
\sum_{j=0}^n\frac{j!(j+\lambda)}{\lambda(2\lambda)_j}C_j^\lambda(x)C_j^\lambda(y)\\
%=&
%\frac{\sqrt{\pi}\Gamma(\lambda+1/2)}{\Gamma(\lambda+1)}
%\sum_{j=0}^n\frac{j!(j+\lambda)\Gamma^2(\lambda)}{\pi 2^{1-2\lambda}\Gamma(j+2\lambda)}
%C_j^\lambda(x)C_j^\lambda(y)\nonumber\\
%&=\frac{\sqrt{\pi}\Gamma(\lambda+1/2)}{\Gamma(\lambda+1)}
%\frac{\frac{2^n\Gamma(n+\lambda)}{n!\Gamma(\lambda)}}{\frac{2^{n+1}\Gamma(n+\lambda+1)}{(n+1)!\Gamma(\lambda)}}
%\frac{n!(n+\lambda)\Gamma^2(\lambda)}{\pi 2^{1-2\lambda}\Gamma(n+2\lambda)}
%\frac{C_{n+1}^\lambda(x)C_n^\lambda(y)-C_n^\lambda(x)C_{n+1}^\lambda(y)}{x-y}
%\nonumber\\
%&=\frac{\sqrt{\pi}\Gamma(\lambda+1/2)}{\lambda}
%\frac{(n+1)!\Gamma(\lambda)}{\pi 2^{2-2\lambda}\Gamma(n+2\lambda)}
%\frac{C_{n+1}^\lambda(x)C_n^\lambda(y)-C_n^\lambda(x)C_{n+1}^\lambda(y)}{x-y}
%\nonumber\\
%&=\frac{\sqrt{\pi}\Gamma(2\lambda)(2\pi)^{1/2}2^{1/2-2\lambda}}{\lambda}
%\frac{(n+1)!}{\pi 2^{2-2\lambda}\Gamma(n+2\lambda)}
%\frac{C_{n+1}^\lambda(x)C_n^\lambda(y)-C_n^\lambda(x)C_{n+1}^\lambda(y)}{x-y}
%\nonumber\\
%&=\frac{\Gamma(2\lambda)}{2\lambda}
%\frac{(n+1)!}{\Gamma(n+2\lambda)}
%\frac{C_{n+1}^\lambda(x)C_n^\lambda(y)-C_n^\lambda(x)C_{n+1}^\lambda(y)}{x-y}
%\nonumber\\
&=
\frac{(n+1)!}{2\lambda(2\lambda)_n}
\frac{C_{n+1}^\lambda(x)C_n^\lambda(y)-C_n^\lambda(x)C_{n+1}^\lambda(y)}{x-y}\label{Nucleoxy}\,.
\end{align}
The formula above is valid for $x\neq y$. The case $x=y$ is obtained by taking limits:
\begin{align}\label{Nucleox}
K_n^\lambda(x,x)=&\frac{(n+1)!}{2\lambda(2\lambda)_n}
\left(C_n^\lambda(x)(C_{n+1}^\lambda)'(x)-C_{n+1}^\lambda(x)(C_n^\lambda)'(x)\right)\nonumber
\\=&\frac{(n+1)!}{(2\lambda)_n}
\left(C_n^\lambda(x)C_{n}^{\lambda+1}(x)-C_{n+1}^\lambda(x)C_{n-1}^{\lambda+1}(x)\right)\,.
\end{align}
The case $\lambda=0$, corresponding to $w(x)=w^0(x)=\pi^{-1}(1-x^2)^{-1/2}$, must be done independently since some of the expressions above make no sense if $\lambda=0$. The orthogonal polynomials are just the classical Chebyshev polynomials after some normalization:

$$\widehat{C}_n(x)=\widehat{C}_n^0(x)=\left.\widehat{C}_n^\lambda(x)\right|_{\lambda=0}
=\left\{\begin{array}{ll}
1,&\qquad\text{if }n=0\,,\\
\sqrt{2}T_n(x)=\sqrt{2}\cos(n\arccos x)\,,&\qquad\text{if }n\geq 1\,.
\end{array}\right.$$
The kernel also admits a simpler expression
\begin{align*}
K_n^0(\cos\theta,\cos\varphi)&=\sum_{j=0}^n\widehat{C}_j^0(\cos\theta)\widehat{C}_j^0(\cos\varphi)=
1+2\sum_{j=1}^n\cos(j\theta)\cos(j\varphi),\\
K_n^0(\cos\theta,\cos\theta)&=
1+2\sum_{j=1}^n\cos^2(j\theta)=1+n+ \frac{\cos((n+1)\theta)\sin(n\theta)}{\sin\theta},
\end{align*}
where for the last sum we have used  \cite[Sec. 1.35]{GR2015}.
According to \eqref{eq:dpps} we aim to compute, for any given $n\geq2$ and $\lambda\in(-1/2,\infty)$, the integrals
\begin{align}
L_1=L_1(\lambda,n)=&\iint_{[-1,1]^2}K_n^\lambda(x,x)K_n^\lambda(y,y)w^\lambda(x)w^\lambda(y)\log\frac{1}{|x-y|}\,d(x,y),\label{eq:L1}\\
L_2=L_2(\lambda,n)=&\iint_{[-1,1]^2}K_n^\lambda(x,y)^2w^\lambda(x)w^\lambda(y)\log\frac{1}{|x-y|}\,d(x,y).\label{eq:L2}
\end{align}
The expected value computed in Theorem \ref{th:main2} is equal to $L_1-L_2$.
In the case $\lambda=0$ we succeed in computing these integrals exactly. In the rest of the cases, we get to an asymptotic value up to $o(n)$.

\section{Proof of Theorem \ref{th:main2} for $\lambda=0$}
We first compute $L_1$ and $L_2$ for the case $\lambda=0$. For simplicity, we omit in this section the dependence on $\lambda$. Note that
\begin{align*}
L_1=&\iint_{[-1,1]^2}K_n(x,x)K_n(y,y)\log\frac{1}{|x-y|} w(x) w(y)\,d(x,y)\\
=&\log 2\iint_{[-1,1]^2}K_n(x,x)K_n(y,y) w(x) w(y)\,d(x,y)\\&+\iint_{[-1,1]^2}K_n(x,x)K_n(y,y)\log\frac{1}{2|x-y|} w(x) w(y)\,d(x,y)\\
=&(n+1)^2 \,\log 2 +\sum_{k,\ell=0}^n\CJ_{k,\ell},
	\end{align*}

where
\[
\CJ_{k,\ell}=\iint_{[-1,1]^2}\widehat C_k(x)^2\widehat C_\ell(y)^2 w(x) w(y)\log\frac{1}{2|x-y|}\,d(x,y).
\]
%Notamos que si $k,\ell\geq1$ entonces
%\[
%\CJ_{k,\ell}=\frac{4}{\pi^2}\iint_{[-1,1]^2}\frac{\cos(k\arccos(x))^2\cos(\ell\arccos(y))^2}{\sqrt{1-x^2}\sqrt{1-y^2}}\log\frac{1}{2|x-y|}\,d(x,y).
%\]
%Por otro lado, si $k=0$ y $\ell\geq1$ entonces
%\[
%\CJ_{0,\ell}=\frac{2}{\pi^2}\iint_{[-1,1]^2}\frac{\cos(\ell\arccos(y))^2}{\sqrt{1-x^2}\sqrt{1-y^2}}\log\frac{1}{2|x-y|}\,d(x,y),
%\]
%y finalmente si $k=\ell=0$ entonces
%\[
%\CJ_{0,0}=\frac{1}{\pi^2}\iint_{[-1,1]^2}\frac{1}{\sqrt{1-x^2}\sqrt{1-y^2}}\log\frac{1}{2|x-y|}\,d(x,y),
%\]
%Estas integrales se calculan de forma muy sencilla como vemos a continuación.
\begin{prop}\label{prop:integral}
	Let $k,\ell\geq0$. Then.
	\[
	\begin{cases}
	\CJ_{k,k}=\frac1{4k},&k\geq1,\\
	\CJ_{k,\ell}=0,&\text{ for any other choice of $k$ and $\ell$.}\\
    \end{cases}
	\]
\end{prop}

\begin{proof}
	We first prove the case $k\neq\ell,\;k,\ell\geq1$. Note that
	\begin{align*}
	\CJ_{k,\ell}=&\frac{4}{\pi^2}\iint_{[-1,1]^2}\frac{\cos^2(k\arccos(x))\cos^2(\ell\arccos(y))}{\sqrt{1-x^2}\sqrt{1-y^2}}\log\frac{1}{2|x-y|}\,d(x,y).
%	\\=&\frac{2}{\pi^2}\int_{\alpha\in [-\pi,\pi]}\,d \alpha \log\frac1{\sqrt{2-2\cos\alpha}}\int_{-\pi}^{\pi}\cos(k\theta)^2\cos(\ell\theta+\ell\alpha)^2\,d\theta.
	\end{align*}
	Applying Lemma \ref{lem:tecnico} with $f(x,y)=\cos^2(k\arccos(x))\cos^2(\ell\arccos(y))$, the inner integral in the right hand side becomes
	\begin{multline*}
	U=\int_{-\pi}^{\pi}\cos^2(k\theta)\cos\left(\ell\theta+\ell \alpha\right)^2\,d\theta=\\\int_{-\pi}^{\pi}\cos^2(k\theta)\left(\cos(\ell\theta)\cos(\ell \alpha)-\sin(\ell\theta)\sin(\ell \alpha)\right)^2\,d\theta,
	\end{multline*}
	then, it suffices to check that $U$ is a constant independent of $\alpha$. Indeed, expanding the term in the parenthesis we easily get
%		\begin{align*}
%	K=&\int_{-\pi}^{\pi}\cos(k\theta)^2\cos(\ell\theta)^2\cos(\ell \alpha)^2\,d\theta\\
%	+&\int_{-\pi}^{\pi}\cos(k\theta)^2\sin(\ell\theta)^2\sin(\ell \alpha)^2\,d\theta\\
%	-2&\int_{-\pi}^{\pi}\cos(k\theta)^2\cos(\ell\theta)\cos(\ell \alpha)\sin(\ell\theta)\sin(\ell \alpha)\,d\theta.
%	\end{align*}
%	Esto es,
	\[
	U=\frac{\pi\cos^2(\ell\alpha)}{2}+\frac{\pi\sin^2(\ell\alpha)}{2}=\frac{\pi}{2}.
	\]
The case $k=0,\ell\geq1$ (equiv. $k\geq1$, $\ell=0$) is very similar. The integral this time is
\[
U=\int_{-\pi}^{\pi}\cos\left(\ell\theta+\ell \alpha\right)^2\,d\theta=\pi,
\]
proving the proposition in that case. The case $k,\ell=0$ is even easier since we are just integrating a constant. The remaining case  is $k=\ell\geq1$, for which we have
\[
U=\int_{-\pi}^{\pi}\cos^2(k\theta)\cos^2\left(k\theta+k \alpha\right)\,d\theta=\frac{\pi}{2}+\frac{\pi}{4}\cos(2k\alpha).
\]
From Lemma \ref{lem:tecnico} we get to
		\begin{align*}
\CJ_{k,k}=&\frac1{2\pi}\int_{-\pi}^{\pi}\cos(2k\alpha) \log\frac1{\sqrt{2-2\cos\alpha}}\,d \alpha\stackrel{\text{Lemma \ref{lem:integralestabla}}}{=}\frac{1}{4k}.
\end{align*}
\end{proof}
As a direct consequence of the previous results we get
\begin{cor}\label{cor:cheb}
If $\lambda=0$, then
\[
L_1=(n+1)^2\log 2+\frac{H_n}{4}.
\]
\end{cor}

\begin{lem}\label{lem:L2}
\begin{align*}
	L_2=&(n+1)\log2+H_n+nH_{2n-1}+\frac{H_{2n}}{2}-n+\frac12.
\end{align*}
\end{lem}
\begin{proof}
Note that
\[
K_n(x,y)^2=\left(\sum_{k=0}^n\widehat C_k(x)\widehat C_k(y)\right)^2=\sum_{k,\ell=0}^n\widehat C_k(x)\widehat C_k(y)\widehat C_\ell(x)\widehat C_\ell(y),
\]
which leads us to $L_2=P_0+P+2Q+2R+S$ where:
\begin{align*}
P_0=&\log2\iint_{[-1,1]^2}K_n(x,y)^2 w(x) w(y)\,d(x,y)=(n+1)\log2,\\
P=&\iint_{[-1,1]^2}\log\frac{1}{2|x-y|} w(x) w(y)\,d(x,y)=\CJ_{0,0}=0,
\\
Q=&\frac{2}{\pi^2}\sum_{k=1}^n\iint_{[-1,1]^2}\frac{\cos(k\arccos(x))\cos(k\arccos(y))}{\sqrt{1-x^2}\sqrt{1-y^2}}\log\frac{1}{2|x-y|}\,d(x,y)\\
R=&\frac{4}{\pi^2}\sum_{k=2}^n\sum_{\ell=1}^{k-1} \iint_{[-1,1]^2}\frac{f_{k,\ell}(x,y)}{\sqrt{1-x^2}\sqrt{1-y^2}}\log\frac{1}{2|x-y|}\,d(x,y)\\
S=&\frac{4}{\pi^2}\sum_{k=1}^n \iint_{[-1,1]^2}\frac{\cos^2(k\arccos(x))\cos^2(k\arccos(y))}{\sqrt{1-x^2}\sqrt{1-y^2}}\log\frac{1}{2|x-y|}\,d(x,y)\\
=&\sum_{k=1}^n\CJ_{k,k}\stackrel{\text{Prop. \ref{prop:integral}}}{=}\frac{H_n}{4},
\end{align*}
and
\[
f_{k,\ell}(x,y)=\cos(k\arccos(x))\cos(k\arccos(y))\cos(\ell\arccos(x))\cos(\ell\arccos(y)).
\]
From Lemma \ref{lem:tecnico} we have for $n\geq1$:
\begin{multline*}
Q=\sum_{k=1}^n\frac1{\pi^2}\int_{-\pi}^{\pi}\,d \alpha \log\frac1{\sqrt{2-2\cos\alpha}}\int_{-\pi}^{\pi}\cos(k\theta)\cos(k\theta+k\alpha)\,d\theta\\
=\sum_{k=1}^n\frac{1}{\pi}\int_{-\pi}^{\pi}\cos (k\alpha) \log\frac1{\sqrt{2-2\cos\alpha}} \,d \alpha
\stackrel{\text{Lemma \ref{lem:integralestabla}}}{=}\sum_{k=1}^n\frac{1}{k}=H_n.
\end{multline*}
On the other hand, also from Lemma \ref{lem:tecnico},
\begin{multline*}
R=\sum_{k=2}^n\sum_{\ell=1}^{k-1}\frac2{\pi^2}\int_{-\pi}^{\pi}\,d \alpha \log\frac1{\sqrt{2-2\cos\alpha}}\int_{-\pi}^{\pi}\cos(k\theta)\cos(k\theta+k\alpha)\cos(\ell\theta)\cos(\ell\theta+\ell\alpha)\,d\theta\\
=\sum_{k=2}^n\sum_{\ell=1}^{k-1}\frac1{\pi}\int_{-\pi}^{\pi}\cos(k\alpha)\cos(\ell \alpha) \log\frac1{\sqrt{2-2\cos\alpha}}\,d \alpha.
%\\
%=\frac2{\pi}\int_{\alpha\in [0,\pi]}\left(\sum_{k=1}^n\cos(k\alpha)\right)^2 \log\frac1{\sqrt{2-2\cos\alpha}}\,d \alpha
%\\
%=\frac2{\pi}\int_{\alpha\in [0,\pi]}\left(\frac{\sin\left(n\alpha+\alpha/2\right)}{2\sin(\alpha/2)}-\frac12\right)^2 \log\frac1{2\sin(\alpha/2)}\,d \alpha
%\\
%=\frac2{\pi}\int_{0}^{\pi/2}\left(\frac{\sin\left(2n\theta+\theta\right)}{\sin(\theta)}-1\right)^2 \log\frac1{2\sin(\theta)}\,d \theta
%\\
%=\frac2{\pi}\int_{0}^{\pi/2}\left(\mathcal{U}_{2n}(\cos\theta)-1\right)^2 \log\frac1{2\sin(\theta)}\,d \theta,
\end{multline*}
We have computed these integrals in Lemma \ref{lem:integralestabla} getting:
\begin{align*}
	2R=&\sum_{k=2}^n\sum_{l=1}^{k-1}\left(\frac{1}{k-\ell}+\frac{1}{k+\ell}\right)\\
	=&\sum_{k=2}^n\left(H_{2k-1}-\frac{1}{k}\right)=1- H_n+\sum_{k=2}^nH_{2k-1}\\\stackrel{\text{Lemma \ref{lem:harmonicsum}}}=&nH_{2n-1}+\frac{H_{2n}}{2}-\frac{5H_n}{4}-n+\frac12.
\end{align*}
All in one, we have proved that
\begin{align*}
L_2=&(n+1)\log2+H_n+nH_{2n-1}+\frac{H_{2n}}{2}-n+\frac12,
\end{align*}
as wanted.
\end{proof}
The proof of Theorem \ref{th:main2} is now complete for $\lambda=0$ since we just need to write down $L_1-L_2$. The asymptotic expression for $\mathrm E(0,n+1)$ is obtained from the one for the harmonic number.

\section{Proof of Theorem  \ref{th:main2} for general $\lambda\neq0$}
We first compute $L_1$.  In this section, $Q(\lambda)$ is some constant depending only on $\lambda$, its value may change from one appareance to another and we do not care about it.
\subsection{The value of $L_1$}
In this section we prove the following result
\begin{prop}\label{prop:L1genlambda}
For any $\lambda>-1/2$ and $n\geq2$,
\[
(n+1)^2\log2\leq L_1(\lambda,n)\leq {(n+1)^2}\log2+o(n).
\]
\end{prop}
In order to prove this result, we need some simple lemmas first. Recall from \cite[Eq. (7.33.6)]{S1975} the following: for $n\geq 1$ and $\lambda\in(-1/2,\infty)$,
\begin{equation}
\left|C_n^\lambda(\cos\theta)\right|\leq\begin{cases}Q(\lambda)\sin^{-\lambda}(\theta)n^{\lambda-1}, &cn^{-1}\leq\theta\leq \pi-cn^{-1}\\Q(\lambda)n^{2\lambda-1},&0\leq\theta\leq cn^{-1}\end{cases},
\label{eq:Szegobound}
\end{equation}
for $c>0$ fixed and some constant $Q(\lambda)$. Observe that when $\lambda\geq 0$, both inequalities, though less sharp, hold in $[0,\pi]$. The following lemma follows easily:

\begin{lem}\label{lem:boundChat}
Let $\lambda>-1/2$ and $n\geq1$. Then, for $x\in[-1,1]$ such that $0\leq \arccos x\leq \sqrt{2}\,n^{-1}$ (this holds in particular if $1-n^{-2}\leq x\leq 1$), we have
\[
K_n^\lambda(x,x)w^\lambda(x) \leq
  Q(\lambda)n^{2\lambda+1}(1-x^2)^{\lambda-1/2}.
\]
\end{lem}

\begin{proof}
%Due to symmetry, it suffices to consider $\theta\in[0,\pi/2]$.
Recall that in the range of $x=\cos\theta$,
\[
\hat C_k^\lambda(\cos\theta)^2=\frac{k!(k+\lambda)\Gamma(2\lambda)}{\lambda\Gamma(2\lambda+k)} C_k^\lambda(\cos\theta)^2\leq \frac{Q(\lambda)}{k^{2\lambda-2}}C_k^\lambda(\cos\theta)^2
\stackrel{\text{\eqref{eq:Szegobound}}}{\leq}Q(\lambda)k^{2\lambda},
\]
where we have used some standard estimates on the Gamma function. Then,
\[
K_n^\lambda(x,x)=\sum_{k=0}^n\hat C_k^\lambda(\cos\theta)^2
\leq  Q(\lambda) \sum_{k=1}^nk^{2\lambda}.
\]
For all $\lambda\in(-1/2,\infty)$, the sum is bounded above by $Q(\lambda)n^{2\lambda+1}$, and we are done.
% Now, we have two cases:
%\begin{itemize}
%%\item If $n^{-1}\leq \theta\leq \pi/2$ then from \eqref{eq:Szegobound} we conclude
%%\[
%%\hat C_n^\lambda(\cos\theta)^2\sin^{2\lambda}\theta\leq \frac{Q(\lambda)n^{2\lambda-2}}{n^{2\lambda-2}}\left(\frac{\sin\theta}{\theta}\right)^{2\lambda}\leq Q(\lambda).
%%\]
%\item If $\lambda\geq0$ then from \eqref{eq:Szegobound} we have
%\[
%\hat C_n^\lambda(\cos\theta)^2\sin^{2\lambda}\theta\leq Q(\lambda)n^{2\lambda}\sin^{2\lambda}\theta\leq  Q(\lambda)n^{2\lambda}n^{-2\lambda}\leq Q(\lambda).
%\]
%Seting $x=\cos\theta$, we therefore have
%\[
%K_n^\lambda(x,x)w^\lambda(x)\sqrt{1-x^2}=Q(\lambda)\sum_{k=0}^n\hat C_k^\lambda(\cos\theta)^2\sin^{2\lambda}\theta\leq Q(\lambda)(n+1).
%\]
%\item If $\lambda<0$ then again from \eqref{eq:Szegobound} we have
%\[
%\hat C_n^\lambda(\cos\theta)^2\sin^{2\lambda}\theta\leq Q(\lambda)n^{2\lambda}\sin^{2\lambda}\theta,
%\]
%and hence by the same reasoning as above,
%\[
%K_n^\lambda(x,x)w^\lambda(x)\sqrt{1-x^2}\leq
%Q(\lambda)\sin^{2\lambda}\theta\sum_{k=1}^nk^{2\lambda}\leq
%Q(\lambda)\sin^{2\lambda}\theta=Q(\lambda)\sqrt{1-x^2}^{\,\lambda},
%\]
%\end{itemize}
%since $\sum_{k=1}^\infty k^{2\lambda}=Q(\lambda)<\infty$.
\end{proof}

The following is almost inmediate:
\begin{lem}\label{LemaAcotV}
Let $n>1$. There exists a constant $c>0$ such that for any $y\in[-1,1]$:
$$\int_{-1}^{1}\left|\frac{\log|x-y|}{\sqrt{1-x^2}}\right|dx\leq 3\pi\log 2,\quad \int_{-1+n^{-2}}^{1-n^{-2}}\left|\frac{\log|x-y|}{1-x^2}\right|dx\leq c\sqrt{n}.$$
If additionally we have $\lambda\in(-1/2,0)$, then
\[
\int_{1-n^{-2}}^{1}\left|\frac{\log|x-y|}{\sqrt{1-x^2}^{1-2\lambda}}\right|\,dx\leq Q(\lambda)n^{-1-2\lambda+(2+4\lambda)/(3-2\lambda)},
\]
and
\[
\int_{1-n^{-2}}^{1}\frac{1}{\sqrt{1-x^2}^{1-2\lambda}}\,dx\leq Q(\lambda)n^{-1-2\lambda}.
\]
\end{lem}
\begin{proof}
Recall that
\[
\int_{-1}^1\frac{1}{\pi\sqrt{1-x^2}}dx=1,\quad \int_{-1}^1\frac{\log|x-y|}{\pi\sqrt{1-x^2}}dx=-\log 2,\quad\forall y\in[-1,1].
\]
This last integral is consequence of the potential of the equilibrium measure, also known as Robin constant. Since for any $x,y\in[-1,1]$ we have $\log\frac{|x-y|}{2}\leq0$, we conclude
\begin{align*}
\int_{-1}^1\left|\frac{\log|x-y|}{\pi\sqrt{1-x^2}}\right|dx=&
\int_{-1}^1\left|\frac{\log2+\log\frac{|x-y|}{2}}{\pi\sqrt{1-x^2}}\right|dx\\
\leq & \int_{-1}^1\frac{\log2}{\pi\sqrt{1-x^2}}dx-\int_{-1}^1\frac{\log\frac{|x-y|}{2}}{\pi\sqrt{1-x^2}}dx
\\
= & 2\int_{-1}^1\frac{\log2}{\pi\sqrt{1-x^2}}dx-\int_{-1}^1\frac{\log|x-y|}{\pi\sqrt{1-x^2}}dx = 3\log 2.
\end{align*}
The first claim of the lemma follows. For the second one, let $J$ be the integral we want to estimate. From Holder's inequality,
\[
J\leq \left(\int_{-1+n^{-2}}^{1-n^{-2}}\left|\log|x-y|\right|^4dx\right)^{1/4}\left(\int_{-1+n^{-2}}^{1-n^{-2}}\frac{1}{(1-x^2)^{4/3}}dx\right)^{3/4}.
\]
The first of these two integrals is bounded above by some universal constant. The second one is at most
\[
\left(2\int_{-1+n^{-2}}^0\frac{1}{(1+x)^{4/3} }dx\right)^{3/4}\leq c \sqrt{n}.
\]

For the third integral in the lemma (that we denote $I(\lambda,n)$), again Holder's inequality yields the upper bound (independent of $y$):
\begin{multline*}
I(\lambda,n)\leq Q(q,\lambda)\left(\int_{1-n^{-2}}^{1}\frac{1}{\sqrt{1-x^2}^{(1-2\lambda)q}}\,dx\right)^{1/q} \\
\leq Q(q,\lambda)\left(\int_{1-n^{-2}}^{1}\frac{1}{(1-x)^{(1/2-\lambda)q}}\,dx\right)^{1/q}\leq Q(q,\lambda)n^{1-2\lambda-2/q},
\end{multline*}
where $q>1$ is any positive number such that $(1-2\lambda)q<2$ and $Q(q,\lambda)$ plays the same role as $Q(\lambda)$ but in this case it may depend on $q$. Indeed we choose $q=(3-2\lambda)/(2-4\lambda)$ and the third inequality follows. The last one is even more elementary:
\[
\int_{1-n^{-2}}^{1}\frac{1}{\sqrt{1-x^2}^{1-2\lambda}}\,dx\leq Q(\lambda) \int_{1-n^{-2}}^{1}\frac{1}{(1-x)^{1/2-\lambda}}\,dx\leq Q(\lambda)n^{-1-2\lambda}.
\]
\end{proof}

\begin{lem}\label{lem:boundC}
For $\lambda>-1/2$, $n\geq 1$ and $\theta\in[0,\pi]$
\begin{equation}\label{AsGegenbauerSzego2}
\left|\widehat{C}_n^\lambda(\cos\theta)\sqrt{w^\lambda(\cos(\theta))}
-\frac{\sqrt{2}}{\sqrt{\pi}\sqrt{\sin\theta}}\cos\left((n+\lambda)\theta-\lambda \pi/2\right)\right|
\leq \frac{Q(\lambda)}{n\sin^{3/2}\theta}\min(1,T)\,,
\end{equation}
where $Q(\lambda)$ is some constant depending only on $\lambda$ and
\[
T=T(n,\theta)=\begin{cases}
    n\sin\theta, & \mbox{if } \lambda\geq0 \\
    (n\sin\theta)^{\lambda+1}, & \mbox{otherwise}.
  \end{cases}
\]
One can change $\min(1,T)$ to $T/(T+1)$ if desired.

\end{lem}
\begin{proof}
The classical asymptotic results for Gegenbauer polynomials \cite[Eq. (8.21.18)]{S1975} yields:
\[
\left|C_n^\lambda(\cos\theta)-\frac{2^\lambda S_n^\lambda}{\sqrt{\pi n}\sin^\lambda\theta}\cos\left((n+\lambda)\theta-\lambda \pi/2\right)\right|\leq\frac{|S_n^\lambda| Q(\lambda)}{n^{3/2}\sin^{\lambda+1}\theta}\leq \frac{Q(\lambda)n^{\lambda-2}}{\sin^{\lambda+1}\theta},
\]
valid for $\theta\in[n^{-1},\pi-n^{-1}]$, where
\[
S_n^\lambda=\frac{\Gamma(n+2\lambda)\Gamma(\lambda+1/2)}{\Gamma(2\lambda)\Gamma(\lambda+n+1/2)}\text{ satisfies }|S_n^\lambda|\leq Q(\lambda) n^{\lambda-1/2}.
\]
Then, for the orthonormal polynomials we have
\[
\left|\widehat{C}_n^\lambda(\cos\theta)-\frac{\gamma_n^\lambda 2^\lambda S_n^\lambda}{\sqrt{\pi n}\sin^\lambda\theta}\cos\left((n+\lambda)\theta-\lambda \pi/2\right)\right|
\leq\frac{\gamma_n^\lambda Q(\lambda)n^{\lambda-2}}{\sin^{\lambda+1}\theta}
\leq\frac{Q(\lambda)}{n\sin^{\lambda+1}\theta}\,.
\]
Using
$$\left|\frac{\gamma_n^\lambda 2^\lambda S_n^\lambda}{\sqrt{\pi n}}-\sqrt{\frac{2\Gamma(\lambda+1/2)}{\sqrt{\pi}\Gamma(\lambda+1)}}\right|\leq \frac{Q(\lambda)}{n}\,,$$
and multiplying by $\sqrt{w^\lambda}$ we get
\[
\left|\widehat{C}_n^\lambda(\cos\theta)\sqrt{w^\lambda(\cos(\theta))}
-\frac{\sqrt{2}}{\sqrt{\pi}\sqrt{\sin\theta}}\cos\left((n+\lambda)\theta-\lambda \pi/2\right)\right|
\leq\frac{Q(\lambda)}{n\sin^{3/2}\theta}+\frac{Q(\lambda)}{n\sqrt{\sin\theta}}\,,
\]
so \eqref{AsGegenbauerSzego2} is proved for $\theta\in[n^{-1},\pi-n^{-1}]$.

Now if $\theta\in[0,n^{-1}]\cup[\pi-n^{-1},\pi]$, using \eqref{eq:Szegobound} we obtain
\begin{align*}
n\sin^{3/2}(\theta)
&\left|\widehat{C}_n^\lambda(\cos\theta)\sqrt{w^\lambda(\cos(\theta))}
-\frac{\sqrt{2}}{\sqrt{\pi}\sqrt{\sin\theta}}\cos\left((n+\lambda)\theta-\lambda \pi/2\right)\right|\\
\leq &n\sin^{3/2}(\theta)Q(\lambda)\left(\gamma_n^\lambda n^{2\lambda-1} \sin^{\lambda-1/2}(\theta)+\frac{1}{\sqrt{\sin\theta}}\right)\\
\leq &n\sin^{3/2}(\theta)Q(\lambda)\left(n^{\lambda} \sin^{\lambda-1/2}(\theta)+\frac{1}{\sqrt{\sin\theta}}\right)\\
\leq &Q(\lambda)\left( (n \sin(\theta))^{\lambda+1}+n\sin\theta\right),
\end{align*}
which is bounded above by $n\sin\theta$ if $\lambda\geq0$ and by $(n\sin\theta)^{\lambda+1}$ otherwise, so \eqref{AsGegenbauerSzego2} is valid for $\theta\in[0,\pi]$.
\end{proof}

\begin{lem}\label{lem:boundK}
Let $\lambda>-1/2$ and let $n\geq 2$, then
%\[
%\left|\frac{K_n^\lambda(\cos\theta,\cos\theta)w^\lambda(\cos\theta)-\frac{n+1}{\pi\sqrt{1-x^2}}\right|\leq Q(\lambda
%\]
\[
\left|K_{n}^\lambda(x,x)w^\lambda(x)-\frac{n+1}{\pi\sqrt{1-x^2}}\right|\leq \frac{Q(\lambda)\log n}{1-x^2}
,\quad x\in[-1,1]
\]
where $Q(\lambda)$ is some constant depending only on $\lambda$.
\end{lem}
\begin{proof}
For $\theta=\arccos(x)\in[0,\pi]$ and $k\geq 1$,  by Lemma \ref{lem:boundC}
\begin{align*}
&\left|\widehat{C}_k^\lambda(\cos\theta)^2w^\lambda(\cos\theta)
-\frac{2}{\pi\sin\theta}\cos^2((k+\lambda)\theta-\lambda\pi/2)\right|\\
\leq&\frac{Q(\lambda)}{k\sin^{3/2}\theta}\frac{T(k,\theta)}{T(k,\theta)+1}
\left(\left|\widehat{C}_k^\lambda(\cos\theta)\right|\sqrt{w^\lambda(\cos(\theta))}+
\frac{\sqrt{2}}{\sqrt{\pi}\sqrt{\sin\theta}}\right)\\
\leq& \frac{Q(\lambda)}{k\sin^{3/2}\theta}\frac{T(k,\theta)}{ T(k,\theta)+1}\left(\frac{Q(\lambda)}{k\sin^{3/2}\theta}\frac{T(k,\theta)}{T(k,\theta)+1} +\frac{1}{\sqrt{\sin\theta}}\right)\\
\leq &\frac{Q(\lambda)}{(k+1)^2\sin^{3}\theta}\left(\frac{T(k,\theta)}{T(k,\theta)+1}\right)^2
+\frac{Q(\lambda)}{(k+1)\sin^{2}\theta}\left(\frac{T(k,\theta)}{T(k,\theta)+1}\right)\,,
\end{align*}
where we have used Lemma \ref{lem:boundC}  and that the normalization coefficient $\gamma_k^\lambda$ behaves like $k^{1-\lambda}$. Observe that this inequality is also valid for $k=0$. Then, from Lemma \ref{lem:trigb} we conclude
\begin{multline*}
K_n^\lambda(x,x)w^\lambda(x)=\sum_{k=0}^n\widehat{ C}_k^\lambda(x)^2w^\lambda(x)\le\\
\frac{n+1}{\pi\sin\theta}+ \sum_{k=0}^n\frac{Q(\lambda)}{(k+1)^2\sin^{3}\theta}\left(\frac{T(k,\theta)}{T(k,\theta)+1}\right)^2
+\sum_{k=0}^n\frac{Q(\lambda)}{(k+1)\sin^{2}\theta}\left(\frac{T(k,\theta)}{T(k,\theta)+1}\right).%\leq \frac{n+1}{\pi\sin\theta}+\frac{Q(\lambda)\log n}{\sin^{2}\theta},
\end{multline*}
If $\lambda\geq0$ we can bound the last sums by
\begin{align*}
\frac{Q(\lambda)}{\sin\theta}\sum_{k=0}^n\frac{1}{(1+k\sin\theta)^2}\leq & \frac{Q(\lambda)}{\sin\theta}\int_{ 0}^n\frac{1}{(1+x\sin\theta)^2}\,dx\leq \frac{Q(\lambda)}{\sin\theta}\int_0^n\frac{1}{1+x\sin\theta}\,dx,\\
\frac{Q(\lambda)}{\sin\theta}\sum_{k=0}^n\frac{1}{1+k\sin\theta}\leq&
    \frac{Q(\lambda)}{\sin\theta}\int_0^n\frac{1}{1+x\sin\theta}\,dx\leq\frac{Q(\lambda)\log (1+n\sin\theta)}{\sin^2\theta}\\\leq&\frac{Q(\lambda)\log n}{\sin^2\theta}.
\end{align*}
If $\lambda<0$ the corresponding bounds are
\begin{align*}
\frac{Q(\lambda)}{\sin^{1-2\lambda}\theta}\sum_{k=0}^n\frac{k^{2\lambda+2}}{(k+1)^2(1+k^{\lambda+1}\sin^{\lambda+1}\theta)^2} \leq& \frac{Q(\lambda)}{\sin^{1-2\lambda}\theta}\int_{0}^n\frac{x^{2\lambda}}{(1+x^{\lambda+1}\sin^{\lambda+1}\theta)^2}\,dx \\
\leq&
\frac{Q(\lambda)}{\sin^{2}\theta},
\end{align*}
(the last by dividing the integration interval with midpoint $1/\sin\theta$ if $n$ is greater than this quantity), and
\begin{align*}
\frac{Q(\lambda)}{\sin^{1-\lambda}\theta}\sum_{k=0}^n\frac{k^{\lambda+1}}{(k+1)(1+k^{\lambda+1}\sin^{\lambda+1}\theta)} \leq& \frac{Q(\lambda)}{\sin^{1-\lambda}\theta}\int_0^n\frac{x^{\lambda}}{(1+x^{\lambda+1}\sin^{\lambda+1}\theta)}\,dx \\=&
\frac{Q(\lambda)\log (1+(n\sin\theta)^{\lambda+1})}{\sin^2\theta}
\\
\leq&\frac{Q(\lambda)\log n}{\sin^2\theta},
\end{align*}

as we wanted. The reciprocal inequality
\[
K_n^\lambda(x,x)w^\lambda(x)\geq\frac{n+1}{\pi\sin\theta}-\frac{Q(\lambda)\log n}{\sin^{2}\theta}
\]
is proved the same way (now, the error bounds have a minus sign).
\end{proof}

\subsubsection{Proof of Proposition \ref{prop:L1genlambda}}
The integral $L_1$ is indeed the energy of a measure, so we can make use of the terminology related with this topic. The mutual energy of a pair of (possibly signed) measures $\mu$ and $\nu$ is given by
$$I(\mu,\nu)=\iint\log\frac{1}{|x-y|}d\mu(x)d\nu(y)\,,$$
where the integral is assumed to exist, and the energy of a measure $\mu$ is $I(\mu,\mu)$ which is usually denoted simply by $I(\mu)$. The potential of a measure $\mu$ is defined by
$$V^\mu(x)=\int\log\frac{1}{|x-y|}d\mu(y)\,.$$
Thus we have the following obvious relations
$$I(\mu,\nu)=\int V^\mu(x)d\nu(x)=\int V^\nu(x)d\mu(x)\,, \qquad I(\mu)=\int V^\mu(x)d\mu(x)\,.$$

The following well known lemma describes how is the energy of linear transformations of measures (possibly signed).
\begin{lem}\label{lem:linearitymeasures}The following identities hold
$$I(\mu+\nu)=I(\mu)+2I(\mu,\nu)+I(\nu)\,,\qquad I(t \mu)=t^2I(\mu)\,,$$
where $t\in\mathbb{R}$.
\end{lem}
%\begin{proof}
%The proof is quite direct
%\begin{align*}
%I(\mu+\nu)&=\iint\log\frac{1}{|x-y|}d(\mu+\nu)(x)d(\mu+\nu)(y)\\
%&=
%\iint\log\frac{1}{|x-y|}d\mu(x)d\mu(y)+
%\iint\log\frac{1}{|x-y|}d\mu(x)d\nu(y)+
%\iint\log\frac{1}{|x-y|}d\nu(x)d\mu(y)\\
%&\quad+\iint\log\frac{1}{|x-y|}d\nu(x)d\nu(y)\\
%&=I(\mu)+2I(\mu,\nu)+I(\nu)\,,
%\end{align*}
%and on the other side
%$$I(t\mu)=\int\int\log\frac{1}{|x-y|}d(t\mu)(x)d(t\mu)(y)=t^2\int\int\log\frac{1}{|x-y|}d\mu(x)d\mu(y)=t^2I(\mu)$$
%\end{proof}

The equilibrium measure in $[-1,1]$, which from now on we denote by $\mu$, is the unique minimizer of the energy among all the measures of total mass $1$ and also is the unique (unitary) measure such that its potential is constant,  $V^{\mu}(x)=\log(2)$ for $x\in[-1,1]$. Its density is
\begin{equation}\label{MedEq}
\frac{d\mu(x)}{dx}=\frac{1}{\pi\sqrt{1-x^2}}\,.
\end{equation}
In other words, for any other probability measure $d\nu$, the logarithmic energy given by
\[
\iint_{x,y\in[-1,1]}\log\frac{1}{|x-y|}\,d\nu(x)\,d\nu(y)
\]
is greater than $\log2$. Now, since $(n+1)^{-1}K_n^\lambda(x,x)w^\lambda(x)$ is a probability measure, the lower bound of Proposition \ref{prop:L1genlambda} follows immediately since $I((n+1)\mu)=(n+1)^2\log(2)$. For the upper bound we are going to consider the equilibrium measure $\mu$, the measures $\nu_n$ such that
$$\frac{d\nu_n}{dx}=K_n^\lambda(x)w^\lambda(x)\,,$$
and $\varepsilon_n=\nu_n-(n+1)\mu$. Both $\nu_n$ and $(n+1)\mu$ are positive measures of total mass $(n+1)$, hence $\varepsilon_n$ is a signed measure of total mass $0$.

The first step is to use the decomposition $\nu_n=(n+1)\mu+\varepsilon_n$. Thus
$$L_1(n,\lambda)=I(\nu_n)=I\left((n+1)\mu+\varepsilon_n\right)
=I\left((n+1)\mu\right)+2I((n+1)\mu,\varepsilon_n)+I(\varepsilon_n)\,.$$
%The energy of $(n+1)\mu$ can be easily computed
%$$I((n+1)\mu)
%=(n+1)^2I(\mu)
%=(n+1)^2\int_{-1}^1V^\mu(x)d\mu(x)=(n+1)^2\int_{-1}^1\log(2)d\mu(x)=(n+1)^2\log(2)\,.$$
The mutual energy of $(n+1)\mu$ and $\varepsilon_n$ vanishes since
\begin{align*}
I((n+1)\mu,\varepsilon_n)&=(n+1)I(\mu,\varepsilon_n)=(n+1)\int V^{\mu}(x)d\varepsilon_n(x)
=(n+1)\log(2)\int d\varepsilon_n(x)=0\,.
\end{align*}
The consequence is that we can write
$$L_1(n,\lambda)=(n+1)^2\log(2)+I(\varepsilon_n)\,,$$
and then we have to prove that $I(\varepsilon_n)=o(n)$. Now, from lemmas \ref{lem:boundChat} and \ref{lem:boundK} we have that
\[
|d\varepsilon_n(x)|\leq\begin{cases}\frac{Q(\lambda)\log n}{1-x^2},&x\in[-1+n^{-2},1-n^{-2}],
\\[3pt]
\frac{Q(\lambda)n^{2\lambda+1}(1-x^2)^{\lambda}}{\sqrt{1-x^2}}+\frac{Q(\lambda)n}{\sqrt{1-x^2}}\leq \frac{Q(\lambda)n}{\sqrt{1-x^2}},&x\in[1-n^{-2},1],\;\lambda\geq0
\\[3pt]
\frac{Q(\lambda)n^{2\lambda+1}(1-x^2)^{\lambda}}{\sqrt{1-x^2}}+\frac{Q(\lambda)n}{\sqrt{1-x^2}}\leq \frac{Q(\lambda)n^{2\lambda+1}}{\sqrt{1-x^2}^{1-2\lambda}},&x\in[1-n^{-2},1],\;\lambda<0
\end{cases}
\]
and using a simple symmetry argument we have
\[I(\varepsilon_n)\leq 4A_n+4B_n+C_n
\]
where the definition of $A_n,B_n$ depends on the cases $\lambda\geq0$ and $\lambda<0$. For the first one:
\begin{align*}
A_n=&\iint_{x,y\in[1-n^{-2},1]}\left|\log\frac{1}{|x-y|}\right|\,\frac{Q(\lambda)n^2}{\sqrt{1-x^2}\sqrt{1-y^2}}\,dx\,dy,
\\
B_n=&\iint_{x\in[-1+n^{-2},1-n^{-2}],y\in[1-n^{-2},1]}\left|\log\frac{1}{|x-y|}\right| \frac{Q(\lambda)n\log n}{(1-x^2)\sqrt{1-y^2}}\,dx\,dy,
\\
C_n=&\iint_{x,y\in[-1+n^{-2},1-n^{-2}]}\left|\log\frac{1}{|x-y|}\right|\,\frac{Q(\lambda)(\log n)^2}{(1-x^2)(1-y^2)}\,dx\,dy\\\leq &O(\sqrt{n}\log(n)^3),
\end{align*}
the last from Lemma \ref{LemaAcotV} and some little arithmetic. Also from Lemma \ref{LemaAcotV} we have
\[
B_n\leq Q(\lambda)n^{3/2}\log n\int_{1-n^{-2}}^1\frac{1}{\sqrt{1-y^2}}\,dy \; \leq\;  Q(\lambda)n^{3/2}\log n\sqrt{1-(1-n^{-2})^2}\leq Q(\lambda)n^{1/2}\log n.
\]
Finally, the change of variables $x=1-t$, $y=1-s$ gives
\begin{align*}
A_n\leq&Q(\lambda)n^2\int_{0}^{n^{-2}}dt\int_{0}^t\frac{1}{\sqrt{2t-t^2}\sqrt{2s-s^2}}\log\frac{1}{t-s}\,ds\\
\leq&Q(\lambda)n^2\int_{0}^{n^{-2}}dt\int_{0}^t\frac{1}{\sqrt{ts}}\log\frac{1}{t-s}\,ds\\
=&-Q(\lambda)^2n^2\int_{0}^{n^{-2}}2(\log t+2\log 2-2)dt=O(\log n).\end{align*}
In the case $\lambda\in(-1/2,0)$ the definitions of $A_n$ and $B_n$ are slightly different but they satisfy similar bounds: $A_n$ equals
\[
\iint_{{x,y\in[1-n^{-2},1]}}\left|\log\frac{1}{|x-y|}\right|\,\frac{Q(\lambda)n^{4\lambda+2}}{\sqrt{1-x^2}^{1-2\lambda}\sqrt{1-y^2}^{1-2\lambda}}\,dx\,dy\leq  Q(\lambda) n^{(2+4\lambda)/(3-2\lambda)}=o(n),
\]
where we have used Lemma \ref{LemaAcotV}. Finally, the definition of $B_n$ if $\lambda\in(-1/2,0)$ is
\[
\iint_{x\in[-1+n^{-2},1-n^{-2}] ,y\in[1-n^{-2},1]}\left|\log\frac{1}{|x-y|}\right|\,\frac{Q(\lambda)n^{2\lambda+1}\log n}{(1-x^2)\sqrt{1-y^2}^{1-2\lambda}}\,dx\,dy=O(\sqrt{n}\log n),
%\leq Q(\lambda)\sqrt{n}\log n\int_{1-\carlos{n^{-2}}}^1\frac{1}{\sqrt{1-y^2}^{1+2|\lambda|}}=O(\sqrt{n}\log n).
\]
again from Lemma \ref{LemaAcotV}.

\subsection{The value of $L_2$}

We devote this section to prove the following result.
\begin{prop}\label{propL2}
For any $\lambda>-1/2$, and as $n\to\infty$
$$L_2(\lambda,n)=n\log(n)+(\gamma+2\log(2)-1)n+o(n)\,.$$
\end{prop}
Let us commence the proof with some auxiliary results about the kernel $K_n^{\lambda}(x,y)$, some of them interesting on their own.
\begin{prop}\label{KerBound}
Let us take $x=\cos\theta$ and $y=\cos\sigma$. For any $\lambda>-1/2$, $\alpha\in[0,1]$ and $c>0$:
\begin{itemize}
\item[(i)] For $(x,y)\in[-1+cn^{-2\alpha},1-cn^{-2\alpha}]$,
\begin{align*}
&\left|K_n^\lambda(x,y)\sqrt{w^\lambda(x)}\sqrt{w^\lambda(y)}
-\frac{1}{2\pi\sqrt{\sin\theta}\sqrt{\sin\theta}}\frac{\sin((n+\lambda+1/2)(\theta-\sigma))}{\sin((\theta-\sigma)/2)}\right|\\
\leq& \frac{Q(\lambda,c)n^{\alpha}\log n}{\sqrt{\sin\theta}\sqrt{\sin\sigma}}\,.
\end{align*}
\item[(ii)] For $\lambda\geq 0$ and $x\neq y$ both in $[-1,1]$:
$$\left|K_n^\lambda(x,y)\sqrt{w^\lambda(x)}\sqrt{w^\lambda(y)}\right|
\leq\frac{Q(\lambda)}{|x-y|\sqrt{\sin\theta}\sqrt{\sin\sigma}}\,.$$
In the case $\lambda\in(-1/2,0)$ this last inequality is also valid but only for $x\neq y$ both in $[-1+cn^{-2},1-cn^{-2}]$ and additionally:
$$\left|K_n^\lambda(x,y)\sqrt{w^\lambda(x)}\sqrt{w^\lambda(y)}\right|
\leq\frac{Q(\lambda)(n\sin\theta)^{\lambda}}{|x-y|\sqrt{\sin\theta}\sqrt{\sin\sigma}}\,,$$
for  $(x,y)\in[1-cn^{-2},1]\times[-1+cn^{-2},1-cn^{-2}]$.
\item[(iii)] For $(x,y)\in[-1,1]^2$ and $\lambda\geq 0$
$$\left|K_n^\lambda(x,y)\sqrt{w^\lambda(x)}\sqrt{w^\lambda(y)}\right|
\leq\frac{Q(\lambda)n}{\sqrt{\sin\theta}\sqrt{\sin\sigma}}\,.$$
If $-1/2<\lambda<0$ and $(x,y)\in[1-cn^{-2\alpha},1]^2$, then
$$\left|K_n^\lambda(x,y)\right|
\leq Q(\lambda)n^{2\alpha\lambda+1}\,,$$
while if $(x,y)\in[1-cn^{-2},1]\times[-1+cn^{-2},1-cn^{-2}]$:
$$\left|K_n^\lambda(x,y)\sqrt{w^\lambda(x)}\sqrt{w^\lambda(y)}\right|
\leq\frac{Q(\lambda)n}{\sqrt{\sin\theta}\sqrt{\sin\sigma}}(n\,\sin\theta)^\lambda\,.$$
\end{itemize}
\end{prop}
\begin{proof}
Let us prove {\emph (i)}.
Starting from \eqref{AsGegenbauerSzego2}, we sum from $k=0$ to $n$ %and with the help of lemma \eqref{lem:SumCos}
\begin{align*}
&\left|K_n(x,y)\sqrt{w^\lambda(x)}\sqrt{w^\lambda(y)}-
\sum_{k=0}^n\frac{2\cos\left((k+\lambda)\theta-\lambda\pi/2\right)\cos\left((k+\lambda)\sigma-\lambda\pi/2\right)}
{\pi\sqrt{\sin\theta}\sqrt{\sin\sigma}}
\right|\\
\leq& \sum_{k=0}^n\frac{Q(\lambda)}{k \sin^{3/2}(\theta)}\frac{T(k,\theta)}{1+T(k,\theta)}
\frac{\sqrt{2}}{\sqrt{\pi}\sqrt{\sin\sigma}}\left|\cos((k+\lambda)\sigma-\lambda\pi/2)\right|\\
&\quad
+\sum_{k=0}^n\frac{\sqrt{2}}{\sqrt{\pi}\sqrt{\sin\theta}}\left|\cos((k+\lambda)\theta-\lambda\pi/2)\right|
\frac{Q(\lambda)}{k\sin^{3/2}(\sigma)}\frac{T(k,\sigma)}{1+T(k,\sigma)}
\\
&\quad+\sum_{k=0}^n \frac{Q(\lambda)}{k\sin^{3/2}(\theta)}\frac{T(k,\theta)}{1+T(k,\theta)} \frac{Q(\lambda)}{k\sin^{3/2}(\sigma)}\frac{T(k,\sigma)}{1+T(k,\sigma)}
\\
\leq&\frac{Q(\lambda)}{\sqrt{\sin\theta}\sqrt{\sin\sigma}}
\left(\sum_{k=1}^n\frac{1}{k\sin\theta}
+\sum_{k=1}^n\frac{1}{k\sin\sigma}
+\sum_{k=1}^n\frac{1}{k\sin\theta} \frac{1}{k\sin\sigma}\frac{T(k,\sigma)}{1+T(k,\sigma)}\right)\,.
\end{align*}
For $\lambda\geq0$ we have $T(k,\sigma)=k\sin\sigma$ and
$$
\frac{1}{k\sin\sigma}\frac{T(k,\sigma)}{1+T(k,\sigma)} = \frac{1}{1+k\sin\sigma}<1.
$$
For $\lambda\in(-1/2,0)$ we have $T(k,\sigma)=(k\sin\sigma)^{\lambda+1}$ and
$$
\frac{1}{k\sin\sigma}\frac{T(k,\sigma)}{1+T(k,\sigma)} = \frac{(k\sin\sigma)^\lambda}{1+(k\sin\sigma)^{\lambda+1}}<(k\sin\sigma)^\lambda.
$$
In both cases, when taking $\theta,\sigma\in[cn^{-\alpha},\pi-cn^{-\alpha}]$, the term in the parenthesis is at most like $n^\alpha\log n$. Note that in the  $\lambda\in(-1/2,0)$ case we use that $\alpha(1-\lambda)+\lambda<\alpha$.
Finally, with the help of lemma \ref{lem:SumCos}
\begin{align*}
&\left|K_n(x,y)\sqrt{w^\lambda(x)}\sqrt{w^\lambda(y)}
-\frac{1}{2\pi\sqrt{\sin\theta}\sqrt{\sin\sigma}}\frac{\sin((n+\lambda+1/2)(\theta-\sigma))}{\sin((\theta-\sigma)/2)}
\right|\\
\leq&
\frac{1}{2\pi\sqrt{\sin\theta}\sqrt{\sin\sigma}}\left(
\left|\frac{\sin((n+\lambda+1/2)(\theta+\sigma)-\lambda \pi)}{\sin((\theta+\sigma)/2)}\right|
+\left|\frac{\sin((\lambda-1/2)(\theta+\sigma)-\lambda\pi)}{\sin((\theta+\sigma)/2)}\right|\right.\\
&\quad\left.+\left|\frac{\sin((\lambda-1/2)(\theta-\sigma))}{\sin((\theta-\sigma)/2)}\right|
+Q(\lambda)n^{\alpha}\log n\right)\\
&\leq\frac{1}{2\pi\sqrt{\sin\theta}\sqrt{\sin\sigma}}
\left(Q(\lambda) n^{\alpha}+Q(\lambda) n^{\alpha}+Q(\lambda) +Q(\lambda)n^{\alpha}\log n\right)
\leq Q(\lambda)n^{\alpha}\log n\,,
\end{align*}
and \emph{(i)} is proved.

Inequalities in \emph{(ii)} are consequence of the Christoffel-Darboux summation formula (see \cite{S1975}) and some bounds for the Gegenbauer polynomials. First inequality in  \eqref{eq:Szegobound} holds for all $\theta\in[0,\pi]$ if $\lambda\geq 0$ and also for $\theta\in[n^{-1},\pi-n^{-1}]$ if $-1/2<\lambda<0$, so in these cases
\begin{multline*}\left|K_n^\lambda(x,y)\sqrt{w^\lambda(x)}\sqrt{w^\lambda(y)}\right|=\left|\frac{(n+1)!}{2\lambda(2\lambda)_n}
\frac{C_{n+1}^\lambda(x)C_{n}^\lambda(y)-C_{n}^\lambda(x)C_{n+1}^\lambda(y)}{x-y}\sqrt{w^\lambda(x)}\sqrt{w^\lambda(y)}\right|\\
\leq Q(\lambda)n^{2-2\lambda}\frac{n^{2\lambda-2}\sin^{-\lambda}(\theta)\sin^{-\lambda}(\sigma)
}{|x-y|}
\sin^{\lambda-1/2}(\theta)\sin^{\lambda-1/2}(\sigma)
\leq \frac{Q(\lambda)}{|x-y|\sqrt{\sin\theta}\sqrt{\sin\sigma}}
\end{multline*}
If $-1/2<\lambda<0$ and $(x,y)\in[1-cn^{-2},1]\times[-1+cn^{-2},1-cn^{-2}]$, then by the second inequality in \eqref{eq:Szegobound}
\begin{multline*}
\left|K_n^\lambda(x,y)\sqrt{w^\lambda(x)}\sqrt{w^\lambda(y)}\right|=\left|\frac{(n+1)!}{2\lambda(2\lambda)_n}
\frac{C_{n+1}^\lambda(x)C_{n}^\lambda(y)-C_{n}^\lambda(x)C_{n+1}^\lambda(y)}{x-y}\sqrt{w^\lambda(x)}\sqrt{w^\lambda(y)}\right|\\
\leq Q(\lambda)n^{2-2\lambda}\frac{n^{2\lambda-1}n^{\lambda-1}\sin^{-\lambda}\sigma
}{|x-y|}
\sin^{\lambda-1/2}(\theta)\sin^{\lambda-1/2}(\sigma)
\leq \frac{Q(\lambda)(n\sin\theta)^\lambda}{|x-y|\sqrt{\sin\theta}\sqrt{\sin\sigma}}\,,
\end{multline*}
Then \emph{(ii)} is proved.

Finally \emph{(iii)} is quite direct: if $\lambda\geq 0$, from first inequality in \eqref{eq:Szegobound} (which is also valid for $\theta\in[0,\pi]$)
\begin{multline*}
\left|K_n^\lambda(x,y)\sqrt{w^\lambda(x)}\sqrt{w^\lambda(y)}\right|\leq Q(\lambda)
 \sum_{k=0}^n \left(\gamma_{k}^\lambda\right)^2\left|C_k^{\lambda}(\cos\theta)C_k^{\lambda}(\cos\sigma)\right|
 \sin^{\lambda-1/2}(\theta)\sin^{\lambda-1/2}(\sigma)\\
 \leq Q(\lambda)\sum_{k=0}^nk^{2-2\lambda}\sin^{-\lambda}(\theta) k^{\lambda-1}\sin^{-\lambda}(\sigma) k^{\lambda-1}
 \sin^{\lambda-1/2}(\theta)\sin^{\lambda-1/2}(\sigma)
%\leq \sum_{k=0}^n \frac{Q(\lambda)}{\sqrt{\sin\theta}\sqrt{\sin\sigma}}
\leq\frac{Q(\lambda)n}{\sqrt{\sin\theta}\sqrt{\sin\sigma}}\,,
\end{multline*}
If $-1/2<\lambda<0$, combining both inequalities in \eqref{eq:Szegobound} one gets
$$\left|C_k^\lambda(\cos\theta)\right|\leq Q(\lambda)n^{(\alpha+1)\lambda-1}\,,\qquad \theta\in[0,dn^{-\alpha}]$$
for some $d$. Hence if $x,y\in[1-cn^{-2\alpha},1]$:
\begin{multline*}
\left|K_n^\lambda(x,y)\right|\leq
 \sum_{k=0}^n \left(\gamma_{k}^\lambda\right)^2\left|C_k^{\lambda}(\cos(\theta))C_k^{\lambda}(\cos(\sigma))\right|
 \leq \sum_{k=0}^n Q(\lambda)k^{2-2\lambda}k^{2(\alpha+1)\lambda-2}\leq Q(\lambda)n^{2\alpha\lambda+1}\,,
\end{multline*}
and if $(x,y)\in[1-cn^{-2},1]\times[-1-cn^{-2},1-cn^{-2}]$, again from \eqref{eq:Szegobound}:
\begin{align*}
&\left|K_n^\lambda(x,y)\right|\sqrt{w^{\lambda}(x)}\sqrt{w^{\lambda}(y)}\leq
 \sum_{k=0}^n \left(\gamma_{k}^\lambda\right)^2\left|C_k^{\lambda}(\cos(\theta))C_k^{\lambda}(\cos(\sigma))\right|
 \sin^{\lambda-1/2}(\theta)\sin^{\lambda-1/2}(\sigma)\\
 \leq& \sum_{k=0}^n Q(\lambda)k^{2-2\lambda}k^{2\lambda-1}\sin^{-\lambda}(\sigma)k^{\lambda-1}
 \sin^{\lambda-1/2}(\theta)\sin^{\lambda-1/2}(\sigma)
 =\frac{Q(\lambda)\sin^\lambda(\theta)}{\sqrt{\sin(\theta)}\sqrt{\sin(\sigma)}}
 \sum_{k=0}^nk^{\lambda}\\
 \leq&\frac{Q(\lambda)n(n\,\sin \theta)^\lambda}{\sqrt{\sin(\theta)}\sqrt{\sin(\sigma)}}\,.
\end{align*}
\end{proof}

We now can establish a series of lemmas which reduce the region from where the main terms in the asymptotics of Proposition \ref{propL2} comes.

\begin{lem}\label{lem:rectangle}
For $\alpha\in(1/2,1]$, $\alpha_0\in(0,1/2)$ such that $\alpha+\alpha_0>1$, $c>1$ and $d >0$ \[\iint_{[1-n^{-2\alpha},1]\times[1- d n^{-2\alpha_0},1-c n^{-2\alpha}]} \left|\log\frac{1}{|x-y|}\right| K_n(x,y)^2 w(x) w(y)\,d(x,y)=o(n).\]
\end{lem}
\begin{proof}
Let $I$ be our integral. In this region it is clear that $\left|\log\frac{1}{|x-y|}\right|\le Q \, \log(n)$ where $Q$ is a constant. In the case $\lambda\ge 0$, to estimate $I$, we use {\emph (iii)} in Proposition \ref{KerBound} and
take into account that
$$
\int_{1-n^{-2\alpha}}^1 \frac{1}{\sqrt{1-x^2}}\,dx\leq Q n^{-\alpha}\  \qquad
\int_{1-dn^{-2\alpha_0}}^{1-cn^{-2\alpha}} \frac{1}{\sqrt{1-y^2}}\,dy \le Q n^{-\alpha_0}\,,
$$
for some constant $Q$, in order to get
$$
\begin{aligned}
I&\le %Q(\lambda)\log(n)n^2\iint_{[1-n^{-2},1]\times[1-n^{-\alpha_0},1-n^{2\alpha}]}\frac{1}{\sqrt{1-x^2}\sqrt{1-y^2}} %\,d(x,y)
%\le
 Q(\lambda) \log(n) n^2 \int_{1-n^{-2\alpha}}^1 \frac{1}{\sqrt{1-x^2}}\,dx \int_{1-dn^{-2\alpha_0}}^{1-cn^{-2\alpha}} \frac{1}{\sqrt{1-y^2}} \,dy  \\
&\le  Q(\lambda) \log(n) n^{2}n^{-\alpha}n^{-\alpha_0}=o(n)\,.
\end{aligned}
$$

In the case $\lambda\in(-1/2,0)$ we use the inequality
$$
\int_{1-n^{-2\alpha}}^1 \frac{(n\sqrt{1-x^2})^{2\lambda}}{\sqrt{1-x^2}}\,dx\leq Q(\lambda) n^{2\lambda}\int_{0}^{k n^{-\alpha}} \theta^{2\lambda} \,d\theta \le Q(\lambda) n^{2\lambda(1-\alpha)-\alpha}
$$
and the last one in \emph{(iii)} of Proposition \ref{KerBound} to obtain
$$
\begin{aligned}
I&\le Q(\lambda)\log(n)n^{2}\int_{1-n^{-2\alpha}}^1 \frac{(n\sqrt{1-x^2})^{2\lambda}}{\sqrt{1-x^2}}\,dx \int_{1-dn^{-2\alpha_0}}^{1-cn^{-2\alpha}} \frac{1}{\sqrt{1-y^2}}\,dy \\
& \le Q(\lambda)\log(n)n^{2+2\lambda(1-\alpha)-\alpha-\alpha_0}=o(n)
\end{aligned}
$$
\end{proof}

\begin{lem}\label{LemInt}
Let $n>1$, $\lambda\in(-1/2,0)$, $\alpha\in(1/2,1)$ and $c>0$. For any $y\in[-1,1]$:
\[
\int_{1-cn^{-2\alpha}}^{1}\left|\log|x-y| w^{\lambda}(x)\right| \,dx\leq Q(\lambda)n^{\alpha(1-2\lambda-2/q)}, \quad 1<q<2\alpha
\]
and
\[
\int_{1-cn^{-2\alpha}}^{1}w^{\lambda}(x)\,dx\leq Q(\lambda)n^{\alpha(-1-2\lambda)}.
\]

\end{lem}
\begin{proof}
Using Holder's inequality the first integral is bounded by:
\[
\begin{aligned}
I\le &Q(q,\lambda) \left(\int_{1-cn^{-2\alpha}}^{1}\frac{1}{(1-x^2)^{(1/2-\lambda)q}}dx\right)^{1/q} \\
&\le Q(q,\lambda) \left(\int_{1-cn^{-2\alpha}}^{1} \frac{1}{(1-x)^{(1/2-\lambda)q}}dx\right)^{1/q} =  Q(q,\lambda) n^{\alpha(1-2\lambda-2/q)}
\end{aligned}
\]
where $(1/2-\lambda)q<1$ and we choose $1<q<2\alpha$.
The last integral:
\[
\int_{1-cn^{-2\alpha}}^{1}(1-x^2)^{\lambda-1/2}\,dx \leq Q(\lambda)\int_{1-cn^{-2\alpha}}^{1}(1-x)^{\lambda-1/2}\,dx \leq Q(\lambda)n^{\alpha(-1-2\lambda)}.
\]

\end{proof}

\begin{lem}\label{lem:Cuadrados}
For $\alpha\in(1/2,1]$ and $c>0$, \[\iint_{[1-cn^{-2\alpha},1]^2} \left|\log\frac{1}{|x-y|}\right| K_n(x,y)^2 w(x) w(y)\,d(x,y)=o(n)\].
\end{lem}
\begin{proof}
For $\lambda>0$, using Proposition \ref{KerBound} \emph{(iii)} we get:
\[
K_n(x,y)^2 w(x) w(y)\leq \frac{Q(\lambda) n^{2}}{(1-x^2)^{1/2}(1-y^2)^{1/2}}.
\]
Therefore,
\begin{multline}
\iint_{[1-cn^{-2\alpha},1]^2} \left|\log\frac{1}{|x-y|}\right| K_n(x,y)^2 w(x) w(y)\,d(x,y) \\
\leq Q(\lambda) n^2 \iint_{[1-cn^{-2\alpha},1]^2} \left|\log\frac{1}{|x-y|}\right| \frac{1}{\sqrt{1-x^2}\sqrt{1-y^2}} \,d(x,y) \leq Q(\lambda)n^{2(1-\alpha)} \log(n)
\end{multline}
The last inequality holds in the same logic that the one used to bound $A_n$ in the proof of Proposition \ref{prop:L1genlambda}.

For $\lambda<0$, using again \emph{(iii)} of Proposition \ref{KerBound}, $K_n(x,y)^2\leq Q(\lambda)n^{4\alpha\lambda+2}$, then from Lemma \ref{LemInt}:
\begin{multline}
\iint_{[1-cn^{-2\alpha},1]^2} \left|\log\frac{1}{|x-y|}\right| K_n(x,y)^2 w(x) w(y)\,d(x,y) \\
\leq Q(\lambda) n^{4\alpha\lambda+2} \iint_{[1-cn^{-2\alpha},1]^2} \left|\log\frac{1}{|x-y|}\right| \frac{1}{\sqrt{1-x^2}^{1-2\lambda}\sqrt{1-y^2}^{1-2\lambda}} \,d(x,y) \\
 \leq Q(\lambda) n^{4\alpha\lambda+2} n^{\alpha(1-2\lambda-2/q)}\int_{[1-cn^{-2\alpha},1]} \frac{1}{\sqrt{1-y^2}^{1-2\lambda}} dy \\
 \leq Q(\lambda)n^{2\alpha\lambda+2+\alpha-2\alpha/q}\int_{[1-cn^{-2\alpha},1]} \frac{1}{\sqrt{1-y}^{1-2\lambda}} dy \\
 \leq Q(\lambda)n^{2\alpha\lambda+2+\alpha-2\alpha/q} n^{\alpha(-1-2\lambda)}=Q(\lambda)n^{2(1-\alpha/q)}
\end{multline}
If $1<q<2\alpha$, then  $2(1-\alpha/q)<1$ and the result follows.
\end{proof}

The last regions whose integrals are $o(n)$ at most are going to be of the type
$$T_{n,\alpha}=\{(x,y)\in[-1,1]^2:\arccos y-\arccos x\geq 2 n^{-\alpha}\}\,.$$
The border line of $T_{n,\alpha}$ with the rest of $[-1,1]^2$ is an arc joining $(x,y)=(-1+c_nn^{-2\alpha},-1)$ with $(x,y)=(1,1-c_nn^{-2\alpha})$ along the graphic of a convex increasing function, where $c_n>0$ and $c_n\to 2$ (see Figure~\ref{tikz}).
\begin{lem}\label{Lem:T_nalpha}
Let $\lambda>-1/2$ and $\alpha\in(0,3/4)$, then as $n\to \infty$
$$\left|\iint_{T_{n,\alpha}}\log\frac{1}{|x-y|}K_n^\lambda(x,y)^2w^{\lambda}(x)w^{\lambda}(y)dxdy\right|=o(n)\,.$$
\end{lem}
\begin{proof} We prove it first for $\alpha<1/2$ and $\lambda\geq 0$. The change of variables
\begin{equation}\label{VarChange}
(u,v)=\left(\frac{\theta-\sigma}{2},\frac{\theta+\sigma}{2}\right)=
\left(\frac{\arccos(x)-\arccos(y)}{2},\frac{\arccos(x)+\arccos(y)}{2}\right)\,,
\end{equation}
transforms $T_{n,\alpha}$ into a triangle $T_{n,\alpha}^*=\{(u,v): u\leq -n^{-\alpha},u+v\geq 0, -u+v\leq \pi\}$, whose vertices are $(-\pi/2,\pi/2)$, $(-n^{-\alpha},n^{-\alpha})$ and $(-n^{-\alpha},\pi-n^{-\alpha})$. Also
$$|x-y|=\left|2\sin((\theta+\sigma)/2)\sin((\theta-\sigma)/2)\right|=|2\sin(v)\sin(u)|\,.$$
Then, first inequality in \emph{(ii)} of Proposition \ref{KerBound} implies that the absolute value of our integral is less than
\begin{align*}
& Q(\lambda)
\iint_{T_{n,\alpha}^*}\log(n)\frac{1}{\sin^2(v)\sin^2(u)}d(u,v)
\leq Q(\lambda)\log(n)\iint_{T_{n,\alpha}^*}\frac{1}{v^2(\pi-v)^2 u^2}
d(u,v)\\
&\leq Q(\lambda)\log(n) n^{2\alpha}\leq o(n)\,,
\end{align*}

so we have proved the lemma in this case. If $\lambda\in(-1/2,0)$ and still $\alpha<1/2$, we consider the subsets of $T_{n,\alpha}$
\begin{align*}
A&=\{(x,y)\in T_{n,\alpha}: |x|\leq 1-n^{-2},|y|\leq 1-n^{-2}\}\,,\\
B&=\{(x,y)\in T_{n,\alpha}: x\geq 1-n^{-2},|y|\leq 1-n^{-2}\}\,,\\
C&=\{(x,y)\in T_{n,\alpha}: x\geq 1-n^{-2},y>1-n^{-2}\}\,,
\end{align*}
so that by symmetry
\begin{align*}
&\left|\iint_{T_{n,\alpha}}\log\frac{1}{|x-y|}
K_n^\lambda(x,y)^2w^{\lambda}(x)w^{\lambda}(y)dxdy\right|
\leq \left|\iint_{A}\log\frac{1}{|x-y|}
K_n^\lambda(x,y)^2w^{\lambda}(x)w^{\lambda}(y)dxdy\right|\\
\quad&+2\left|\iint_{B}\log\frac{1}{|x-y|}
K_n^\lambda(x,y)^2w^{\lambda}(x)w^{\lambda}(y)dxdy\right|
+\left|\iint_{C}\log\frac{1}{|x-y|}
K_n^\lambda(x,y)^2w^{\lambda}(x)w^{\lambda}(y)dxdy\right|\,,
\end{align*}
and we have to prove that these integrals are $o(n)$.
The integral in $A$ can be bounded as in the case $\lambda\geq 0$. The integral in $C$ can be proved to be also $o(n)$ at most by using the same arguments than in Lemma \ref{lem:Cuadrados} (it is even easier since the logarithmic term does not appear). By \emph{(ii)} of Proposition \ref{KerBound}, the absolute value of our integral in $B$ is less than
\begin{align*}
&Q(\lambda)\log(n)\iint_{B}\frac{(n\sin(\theta))^{2\lambda}}{|x-y|^2\sqrt{1-x^2}\sqrt{1-y^2}}dxdy  \leq Q(\lambda)\log(n)n^{2\lambda}n^{4\alpha}\iint_B\frac{w^{\lambda}(x)}{\sqrt{1-y^2}}dxdy\\
&\leq Q(\lambda)\log(n)n^{2\lambda+4\alpha}n^{-2\lambda-1}=Q(\lambda)\log(n)n^{4\alpha-1}=o(n)
\end{align*}
since $\alpha<1/2$. Hence we have proved the lemma for $\alpha\in(0,1/2)$ and any $\lambda>-1/2$.

Suppose now $\alpha\in(1/2,3/4)$ (the case $\alpha=1/2$ can be deduced from this case) and take  $\alpha_0\in(0,1/2)$ such that $\alpha+\alpha_0>1$, $2\alpha-\alpha_0<1$ (observe that these conditions are not void since $\alpha<3/4$). Consider the following subsets of $T_{n,\alpha}$: $T_{n,\alpha_0}$ and
\begin{align*}
E&=\{(x,y)\in T_{n,\alpha}:2n^{-\alpha_0}\geq\arccos y-\arccos x,x\geq 1-n^{-2\alpha}\}\,,\\
F&=\{(x,y)\in T_{n,\alpha}:2n^{-\alpha_0}\geq\arccos y-\arccos x,x\leq 1-n^{-2\alpha},y>-1+n^{-2\alpha}\}\,.
\end{align*}
We have already proved that the integral in $T_{n,\alpha_0}$ is  at most  $o(n)$ and lemma \ref{lem:rectangle} shows that the integral in $E$ is also $o(n)$ at most (because $\alpha+\alpha_0>1$ and $E$ is included in the rectangle $[1-n^{-2\alpha},1] \times [1-3n^{-2\alpha_0},1- n^{-2\alpha}]$). The integral in $F$ can be bounded using \emph{(i)} of Proposition \ref{KerBound}:
\begin{align*}
&\left|\iint_{F}\log\left(\frac{1}{|x-y|}\right)K_n^\lambda(x,y)^2w^\lambda(x)w^\lambda(y)dxdy\right|\\
\leq &Q(\lambda)\log(n)\iint_{F}\left(\left|
\frac{1}{\sin((\theta-\sigma)/2)}\right|
+n^{\alpha}\log (n)\right)^2\frac{1}{\sqrt{1-x^2}\sqrt{1-y^2}}dxdy\\
\leq &Q(\lambda)n^{2\alpha}\log^3(n)\iint_{F}\frac{1}{\sqrt{1-x^2}\sqrt{1-y^2}}dxdy\leq Q(\lambda)n^{2\alpha}\log^3(n)n^{-\alpha_0}=o(n)\,,
\end{align*}
because we have chosen $\alpha_0>2\alpha-1$. Now the lemma follows by gathering these inequalities and taking into account the symmetry.
\end{proof}

%%%%%%%%%%%%%%%%%%%%%%%%%%%%%%%%%%%%%%%
%
%             dibujo de las regiones
%
%%%%%%%%%%%%%%%%%%%%%%%%%%%%%%%%%%%%%%%%%
\begin{figure}[h]
	\begin{center}
		\begin{tikzpicture}[=>latex,thick,domain=-11:11,xscale=1,yscale=1,scale=0.5]
		%cuadro NE
		\draw[-,line width= 1.1 pt] (5.5,10) -- (5.5,5.5) -- (10,5.5) ;
		\coordinate [label=right:{$1-5n^{-\alpha}$}] (uno) at (10,5.5) ;
		\coordinate [label=left:{\phantom{$1-5n^{-\alpha}$}}] (dos) at (-10,5.5) ;
		\coordinate [label={$T_{n,\alpha}$}] (tn) at (6,-5);
		\coordinate [label={$D_{n,\alpha}$}] (tn) at (0.5,-0.5);
		%cuadro SO
		\draw[-,line width= 1.1 pt] (-5.5,-10) -- (-5.5,-5.5) -- (-10,-5.5) ;
		%cuadro grande
		\draw[-,line width= 1.1 pt] (-10,-10) -- (10,-10) -- (10,10) -- (-10,10) -- (-10,-10);
		%curva superior
		\draw[dotted, line width= 1 pt]  (8,10) .. controls  (0,6) and (-6,0) .. (-10,-8);
		\draw[-,line width= 1.1 pt,postaction={decorate,decoration={text along path,text align=center,raise=-15pt,text={|\normalsize|{${\rm arcos}\, x - {\rm arcos}\, y = 2 n^{-\alpha}$}}}}]  (-8.93,-6) .. controls (-5.65,0)  and (0,5.65) .. (6,8.93);
		%curva inferior
		\draw[dotted,line width= 1 pt]  (-8,-10) .. controls  (0,-6) and (6,0) .. (10,8);	\draw[-,line width= 1.1 pt,postaction={decorate,decoration={text along path,text align=center,raise=7pt,text={|\normalsize|{${\rm arcos}\, x - {\rm arcos}\, y = -2 n^{-\alpha}$}}}}]  (-6,-8.93) .. controls  (0,-5.65) and (5.65,0) .. (8.93,6);
		%curva SO
		\draw[-,line width= 1.1 pt]  (-6,-8.93) .. controls  (-7.5,-8.5) and (-8.5,-7.5) .. (-8.93,-6); %
		\coordinate [label=left:{${\rm arcos}\, x + {\rm arcos}\, y = 2\pi-4 n^{-\alpha}$}] (arco4) at (5,-9);
		%curva NE
		\draw[-,line width= 1.1 pt]  (6,8.93) .. controls  (7.5,8.5) and (8.5,7.5) .. (8.93,6); %
		\coordinate [label=left:{${\rm arcos}\, x + {\rm arcos}\, y = 4 n^{-\alpha}$}] (arco3) at (5,9);
		\end{tikzpicture}
	\end{center}
	\caption{Regions on the square $[-1,1]^2$}
	\label{tikz}
\end{figure}
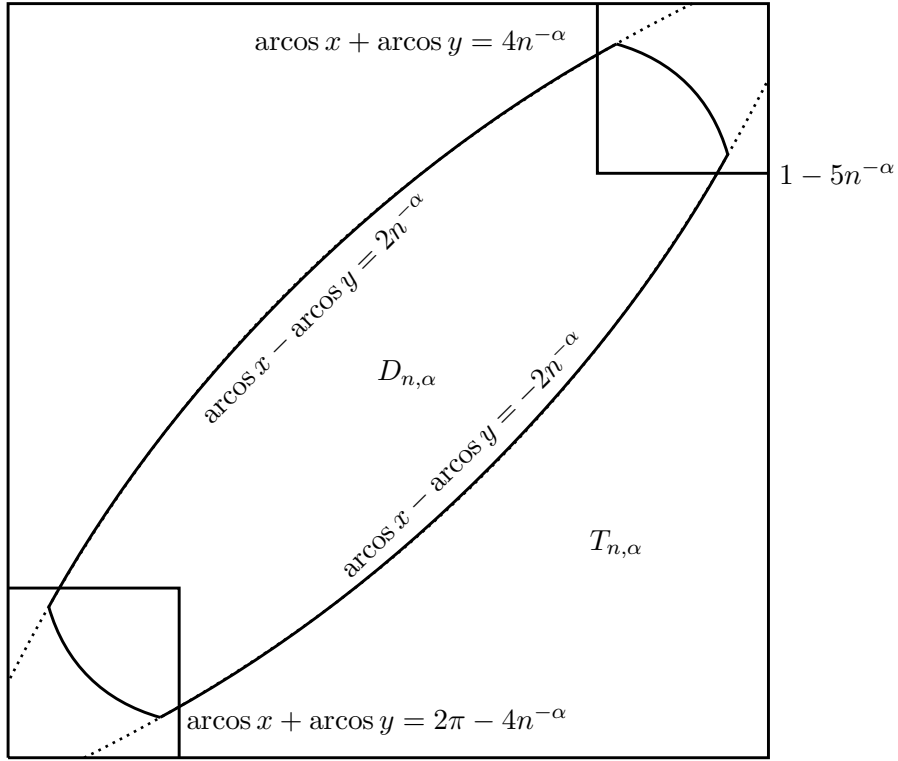

%%%%%%%%%%%%%%%%%%%%%%%%%%%%%%%%%%%%%%%
%
%        fin del     dibujo de las regiones
%
%%%%%%%%%%%%%%%%%%%%%%%%%%%%%%%%%%%%%%%%%

To finish with the estimation of $L_2$, we need to study this integral over the last remaining region, the diagonal
$$
D_{n,\alpha}=\Big\{(x,y)\in[-1,1]^2 \,:\,
\begin{array}{ll}
|\arccos x - \arccos y|\leq 2 n^{-\alpha}
\\
4 n^{-\alpha}\leq \arccos x + \arccos y\leq 2\pi- 4 n^{-\alpha}
\end{array}
\Big\},
$$
which is where the dominant terms will lie. We will devote the rest of this section to prove the following
\begin{prop}
For $\alpha\in(1/2,1)$ it holds
$$
I_{D_{n,\alpha}}=\iint_{D_{n,\alpha}} \log\frac{1}{|x-y|} K_n^\lambda(x,y)^2  \omega^\lambda(x) \omega^\lambda(y) dxdy = n \log n +(-1+\gamma+2\log 2)n+o(n).
$$	
\end{prop}
\begin{proof}
Using Proposition~\ref{KerBound}~\emph{(i)} together with the change of variables \eqref{VarChange} performed in the proof of Lemma~\ref{Lem:T_nalpha}, this integral can be written as
$$
I_{D_{n,\alpha}} = \frac{1}{\pi^2} \iint_{R} \log\frac{1}{2\sin v \sin u} \left[ \frac{\sin\big((2n+2\lambda+1)u\big)}{\sin u} + O\big(n^\alpha \log n\big) \right]^2 dudv,
$$
where the integration region is the rectangle $R=[0, n^{-\alpha}] \times [2n^{-\alpha} , \pi-2n^{-\alpha}]$.
Expanding the squared term, this integral splits into three terms involving three integrals, namely \begin{equation}\label{IntDiag}
I_{D_{n,\alpha}}=I_1 O(n^{2\alpha}\log^2n)+I_2O(n^{\alpha}\log n) + I_3.
\end{equation}
We shall see that the first two terms are $o(n)$ and also that $I_3$ contains all the highest order terms.

We are going to treat each one of these integrals separately. Throughout the proof, we will use a constant $Q$ which is independent on $n$ and might not be the same from one appearing to another.
We start with the first integral
$$
I_1 = \iint_{R} \log\frac{1}{2\sin v\sin u}  dudv,
$$
which can be bounded as
\begin{align*}
|I_1|
& \leq \iint_R \log2 \,  dudv+ \iint_R \log\frac{1}{\sin v}  dudv+ \iint_R \log\frac{1}{\sin u} dudv
\\
& \leq
Q(n^{-\alpha}+ n^{-\alpha}+ n^{-\alpha} \log n )
\leq Q( n^{-\alpha} \log n),
\end{align*}
where we have used that $\log(\sin v)$ is an integrable function on $[0,\pi]$ and also that $\log(1/\sin u)\leq Q \log (1/u)$ for $u\in[0,n^{-\alpha}]$. Then, we directly have that the first term in \eqref{IntDiag} is $o(n)$ since $\alpha<1$.

Now we deal with the second integral
\begin{align*}
I_2 = \iint_{R} \log\frac{1}{2\sin v\sin u} \frac{\sin((2n+2\lambda+1)u)}{\sin u} dudv.
\end{align*}
Observe that the logarithmic term is always positive since $2\sin(v)\sin(u)<1$ on $R$ for $n$ sufficiently large. Then, we can bound
\begin{align*}
|I_2| & \leq \iint_{R} \log\frac{1}{2\sin v} \frac{|\sin((2n+2\lambda+1)u)|}{\sin u} dudv + \iint_{R} \log\frac{1}{\sin u} \frac{|\sin((2n+2\lambda+1)u)|}{\sin u} dudv
\\
& \leq o(1) \int_{0}^{n^{-\alpha}} \frac{|\sin((2n+2\lambda+1)u)|}{\sin u} du + Q \int_{0}^{n^{-\alpha}} \log\frac{1}{\sin u} \frac{|\sin((2n+2\lambda+1)u)|}{\sin u} du,
\end{align*}
where, we have used Lemma~\ref{integ=0} for the estimation of the first term.
Taking into account that we can bound
\begin{equation}\label{bounds1}
\frac{|\sin((2n+2\lambda+1)u)|}{\sin u} \leq Q n, \quad  \log\frac{1}{\sin u} \leq  Q \log\frac{1}{u}, \quad u\in[0,n^{-1}],
\end{equation}
and also
\begin{equation}\label{bounds2}
\frac{|\sin((2n+2\lambda+1)u)|}{\sin u} \leq  Q \frac{1}{u}, \quad  \log\frac{1}{\sin u} \leq  Q \log\frac{1}{u}, \quad u\in[n^{-1}, n^{-\alpha}],
\end{equation}
we split each of the above integrals into these two intervals obtaining
\begin{align*}
|I_2|
& \leq
Q \left(n \int_{0}^{n^{-1}} du
+
 \int_{n^{-1}}^{n^{-\alpha}} \frac{1}{u} du
+
 n \int_{0}^{n^{-1}} \log\frac1u du
+
\int_{n^{-1}}^{n^{-\alpha}} \log\frac{1}{u} \frac{1}{u} du\right)
\\
&\leq
Q \big(1
+
\log n
+
\log n
+
\log^2n \big) \leq Q \log^2n.
\end{align*}
thus, the second term in \eqref{IntDiag} is $o(n)$ since $\alpha<1$.

For the third integral, we split the logarithm factor into two terms and so we can write
\begin{align*}
I_3 &= \frac{1}{\pi^2} \iint_{R} \log\frac{1}{2\sin v\sin u} \frac{\sin^2((2n+2\lambda+1)u)}{\sin^2u} dudv
%\\&
=I_{4}+I_{5}.
\end{align*}
Let us continue with
\begin{align*}
I_4 &= \frac{1}{\pi^2} \iint_{R} \log\frac{1}{2\sin v} \frac{\sin^2((2n+2\lambda+1)u)}{\sin^2u} dudv
%\\&
\leq
o(1) \int_{0}^{n^{-\alpha}}\frac{\sin^2((2n+2\lambda+1)u)}{\sin^2u} du,
\end{align*}
where we have used again Lemma~\ref{integ=0}.
Using the same ideas as in $I_2$, we can split the integration interval, and using the bounds \eqref{bounds1} and \eqref{bounds2} we get
\begin{align*}
I_4 &\leq
o(1) \left( \int_{0}^{n^{-1}} n^2 du
+ \int_{n^{-1}}^{n^{-\alpha}}\frac{1}{u^2} du\right) \leq o(1)n = o(n).
\end{align*}

Finally, we will deal with the remaining integral,
\begin{align}
I_5&=\frac{1}{\pi^2} \iint_{R} \log\frac{1}{\sin u} \frac{\sin^2\big((2n+2\lambda+1)u\big)}{\sin^2u} dudv
\\&
=\frac{\pi-4n^{-\alpha}}{\pi^2} \int_{0}^{n^{-\alpha}}\log\frac{1}{\sin u} \frac{\sin^2\big((2n+2\lambda+1)u\big)}{\sin^2u} du.
\label{eq:I4v0}
\end{align}

We are not going to find a bound, since in this case we are interested in getting exactly the highest order term. First, using Taylor expansion of $u/\sin u$ around $0$, we have
$$
\frac{1}{\sin u} = \frac{1}{u}\big(1+O(n^{-2\alpha})\big), \quad u\in[0,n^{-\alpha}],
$$
from where
$$
\frac{1}{\sin^2u} = \frac{1}{u^2}\big(1+O(n^{-2\alpha})\big),
\quad %u\in[0,n^{-\alpha}],
\log\frac{1}{\sin u} =
%\log\frac{1}{u}+\log\big(1+O(n^{-2\alpha})\big)=
\log\frac{1}{u}+O(n^{-2\alpha}),\quad u\in[0,n^{-\alpha}].
$$
We plug these last two identities in \eqref{eq:I4v0}, getting
\begin{equation}\label{I4-2}
\begin{split}
I_5=&\frac{1}{\pi}\big(1+O(n^{-\alpha})\big)  \int_{0}^{n^{-\alpha}} \log\frac{1}{u} \frac{\sin^2((2n+2\lambda+1)u)}{u^2} du
\\&
+ O(n^{-2\alpha}) \int_{0}^{n^{-\alpha}} \frac{\sin^2\big((2n+2\lambda+1)u\big)}{u^2} du.
\end{split}
\end{equation}
These two integrals in \eqref{I4-2} can be easily computed using the Sine Integral function, defined as
\begin{equation}
\textrm{Si}(z)=\int_{0}^{z}\frac{\sin(t)}{t}dt,
\end{equation}
which has the property
\begin{equation}\label{eq:DSi}
\frac{d}{du} \left( t\,\textrm{Si}(2tu)-\frac{\sin^2(tu)}{u} \right) = \frac{\sin^2(tu)}{u^2}
\end{equation}
and it satisfies the asymptotics (see \cite[(5.2.8), (5.2.34), (5.2.35)]{AS1964})
\begin{equation}\label{SiAsymp}
\textrm{Si}(x) = \frac{\pi}{2}+O\left(\frac{1}{x}\right) \textnormal{ \ as \ } x\to +\infty.
\end{equation}
With these two properties, the integral on the second term of \eqref{I4-2} becomes,
\begin{align}
\int_{0}^{n^{-\alpha}} & \frac{\sin^2 \big((2n+2\lambda+1)u\big)}{u^2} du
\nonumber
\\&
=(2n+2\lambda+1) \, \textrm{Si}\big(2(2n+2\lambda+1) n^{-\alpha}\big)
-\frac{\sin^2\big((2n+2\lambda+1) n^{-\alpha}\big)}{ n^{-\alpha}}
\nonumber
\\&=
(2n+2\lambda+1) \, \big(\pi/2+O(n^{\alpha-1})\big)
+O(n^{\alpha})=n\pi + O(n^{\alpha}).
\label{sin2/u2}
\end{align}

Now, for the integral on the first term of \eqref{I4-2}, we perform integration by parts using \eqref{eq:DSi}, followed by simplification with \eqref{sin2/u2} and the asymptotic of the Sine Function, obtaining
\begin{equation}\label{eq:I4-1}
\begin{split}
\int_{0}^{n^{-\alpha}}  \log\frac{1}{u} & \frac{\sin^2\big((2n+2\lambda+1)u\big)}{u^2}  du = \alpha\pi n\log n - \pi n
\\
&+ (2n+2\lambda+1) \int_0^{n^{-\alpha}}  \frac{\textrm{Si}\big(2(2n+2\lambda+1)u\big)}{u}du
+O(n^\alpha\log n).
\end{split}
\end{equation}
This integral above can be computed using Lemma~\ref{senointegAsymp}, which gives
\begin{equation}\label{lem33}
\int_0^{n^{-\alpha}}  \frac{\textrm{Si}\big(2(2n+2\lambda+1)u\big)}{u}du =  \frac{\pi}{2}\log\big(2(2n+2\lambda+1)n^{-\alpha}\big) + \frac{\gamma \pi}{2} +O(n^{\alpha-1}).
\end{equation}
Then, with \eqref{sin2/u2}, \eqref{eq:I4-1} and \eqref{lem33} plugged into \eqref{I4-2}, and after some straightforward computations we get
\begin{align*}
I_5=& n \log n + (-1+\gamma+2\log2)n+O(n^{1-\alpha}\log n).
\end{align*}

Now the result is proved since $\alpha>1/2$.
\end{proof}

\section{Proof of Corollary \ref{cor:puroymixto}}
Let us consider the process described in the second point of Corollary \ref{cor:puroymixto}.
%For notational reasons it is easier in the proof to generate $n+3$ points: $n+1$ of them coming from the determinantal point process and $2$ extra in te extremes.
 The value of the expected energy for this case is clearly $-2\log 2+L_1-L_2+2L_3$ where $L_1$ and $L_2$ are given in \eqref{eq:L1} and \eqref{eq:L2}, and $L_3$ accounts for the energy corresponding to the crossed terms of the DPP with $n+1$ points and the extremes. In other words, $L_3$ is as given in the following result.
\begin{thm}\label{TeoL3}
Let $L_3=L_3(\lambda,n)$ be defined by
\begin{multline*}
L_3=\int_{-1}^1 K_n^\lambda(x,x)\log\left(\frac{1}{1+x}\right)w^{\lambda}(x)dx+\int_{-1}^1 K_n^\lambda(x,x)\log\left(\frac{1}{1-x}\right)w^{\lambda}(x)dx=\\\int_{-1}^1 K_n^\lambda(x,x)\log\left(\frac{1}{1-x^2}\right)w^{\lambda}(x)dx.
\end{multline*}
Then,
\begin{align*}
L_3=&(n+1)\left(\psi(n+\lambda+1)-\psi(\lambda+1/2)\right)\\
&-(n+2\lambda)
\big(\psi(n+\lambda+1/2)-\psi(\lambda+1/2)-2\psi(2n+2\lambda+1)+2\psi(n+2\lambda+1)\big),
\end{align*}
where $\psi$ is the digamma function.
In particular, for any fixed $\lambda\in(-1/2,\infty)$, we have
$$L_3=2n\log(2)-(2\lambda-1)\log n+O(1).
% - \lambda(2-4 \log2)+(2\lambda-1)\psi(\lambda+1/2)
%-\frac{6\lambda^2-2\lambda-1}{2n}+o(1/n)\,.
$$
\end{thm}
We prove this theorem later. First, let us finish the proof of Corollary \ref{cor:puroymixto}. From Theorem \ref{th:main2}, the energy of the first process in the corollary (generate $n+3$ points with the DPP) is
\begin{multline}\label{eq:aux1}
L_1(\lambda,n+2)-L_2(\lambda,n+2)=\\(n+3)^2\log2-(n+3)\log (n+3)+(1-\gamma-2\log 2)(n+3)+o(n).
\end{multline}
From the same theorem and from Theorem \ref{TeoL3}, the energy of the second process described in Corollary \ref{cor:puroymixto} is
\begin{multline}\label{eq:aux2}
-2\log 2+L_1(\lambda,n)-L_2(\lambda,n)+2L_3(\lambda,n)=\\(n+1)^2\log2-(n+1)\log (n+1)+(1-\gamma-2\log 2)(n+1)+4n\log(2)+o(n).
\end{multline}
It is an easy exercise to check that \eqref{eq:aux1} and \eqref{eq:aux2} describe the same asymptotics up to $o(n)$, as claimed by the corollary. In the case $\lambda=0$ we have exact values for $L_1,L_2$ and $L_3$ so we can compare directly the expresions
\begin{itemize}
\item $L_1(0,n+2)-L_2(0,n+2)$, and
\item $-2\log2+L_1(0,n)-L_2(0,n)+2L_3(0,n)$.
\end{itemize}
It is straightforward to see that the first process has smaller energy than the second one, and that the difference is in the $O(\log n)$ term. This finishes the proof of Corollary \ref{cor:puroymixto} and it only remains to prove Theorem \ref{TeoL3}.

\subsection{Proof of Theorem \ref{TeoL3}}
Note that
\begin{equation}\label{eq:L3comosuma}
L_3=-2\sum_{k=0}^n\int_{-1}^1\widehat{C}^\lambda_k(x)^2w^\lambda(x)\log(1-x)\,dx.
\end{equation}
The case $\lambda=0$ can be done directly. We first compute:
\begin{lem}\label{lem:auxn}
Let $\lambda=0$. Then,
	\[
	\sum_{k=0}^n\int_{-1}^1\widehat{C}^0_k(x)^2w^0(x)\log\frac{1}{\sqrt{2-2x}}\,dx=\frac{H_n}{4}.
	\]
\end{lem}
\begin{proof}
	With the change of variables $x=\cos\theta$, the integral of the lemma becomes
	\begin{multline*}
	\frac1{2\pi}\sum_{k=0}^n\int_{-\pi}^{\pi}\widehat{C}^0_k(\cos\theta)^2\log\frac{1}{\sqrt{2-2\cos\theta}}\,d\theta=\\\frac1{2\pi}\int_{-\pi}^{\pi}\log\frac{1}{\sqrt{2-2\cos\theta}}\,d\theta+\sum_{k=1}^n\frac1{\pi}\int_{-\pi}^{\pi}\cos^2(k\theta)\log\frac{1}{\sqrt{2-2\cos\theta}}\,d\theta=\frac{H_n}{4},
	\end{multline*}
the last from Lemma \ref{lem:integralestabla}.
%\begin{align*}
%L_3=&\frac{4}{\pi}\int_{-1}^1\frac{\log\frac{1}{\sqrt{1-x}}}{\sqrt{1-x^2}}\,dx+\frac{\sqrt{2}}{\pi}\sum_{k=1}^n\int_{-1}^1\frac{\cos(k\arccos(x))^2}{\sqrt{1-x^2}}\log\frac{1}{1-x^2}\,dx\\
%=&\frac{1}{\pi}\int_{0}^\pi\log\frac{1}{\sin^2\theta}\,d\theta+\frac{2}{\pi}\sum_{k=1}^n\int_{0}^\pi\cos(k\theta)^2\log\frac{1}{\sin^2\theta}\,d\theta\\
%\stackrel{\text{Lema \ref{lem:otrasintegrales}}}{=}&\log4+(2\log 2)n+\frac{1}{2}H_n.
%\end{align*}
\end{proof}

%\subsubsection{Proof of Theorem \ref{TeoL3} for $\lambda=0$}
	Note that, for $\lambda=0$
\begin{align*}
	-2\log(1-x)=2\log 2+4\log\frac{1}{\sqrt{2-2x}},
\end{align*}
and combine \eqref{eq:L3comosuma} with Lemma \ref{lem:auxn} to get
\[
L_3=2(n+1)\log2+H_n.
\]
It is easy to check that this equals the expresion in Theorem \ref{TeoL3} (use \cite[6.3.8]{AS1964}).
\begin{flushright}
{$\square$}
\end{flushright}
Now we point at the case $\lambda\neq0$. We will use the following result.
\begin{lem}\label{lem:auxL3previo}
Let $k\geq0$ be an integer. Let $s\in[0,1)$ and $\lambda\in(-1/2,0)\cup(0,\infty)$. Then,
%The following equality holds for all $\lambda>-1/2$, $s\in[0,1)$ and integer $k\geq0$:
\begin{align*}
\int_{-1}^1C^\lambda_k(x)^2(1-x^2)^{\lambda-1/2}(1-x)^s\,dx=&\frac{2^{s+2\lambda}\Gamma(s+\lambda+1/2)\Gamma(\lambda+1/2)\Gamma(k-s)\Gamma(k+2\lambda)^2}{(k!)^2\Gamma(-s)\Gamma(2\lambda+s+k+1)\Gamma(2\lambda)^2}\times\\
&\sum_{\ell=0}^k\frac{(-k)_\ell(k+2\lambda)_\ell(s+\lambda+1/2)_\ell(s+1)_\ell}{\ell!(\lambda+1/2)_\ell(2\lambda+s+k+1)_\ell(s-k+1)_\ell},
\end{align*}
where if $s=0$ the formula is defined by continuation and gives the (expected) value
\[
\int_{-1}^1C^\lambda_k(x)^2(1-x^2)^{\lambda-1/2}\,dx=\frac{\pi2^{1-2\lambda}\Gamma(2\lambda+k)}{k!(k+\lambda)\Gamma(\lambda)^2}.
\]
\end{lem}
\begin{proof}
We rewrite our integral as
\[
\int_{-1}^1C^\lambda_k(x)^2(1-x)^{s+\lambda-1/2}(1+x)^{\lambda-1/2}\,dx,
\]
and find the value of this last integral in \cite[7.314(7)]{GR2015} {(but note there is a typo in that reference!)}:
%\begin{align*}
%I=&\frac{2^{s+2\lambda}\Gamma(s+\lambda+1/2)\Gamma(\lambda+1/2)\Gamma(k-s)\Gamma(k+2\lambda)^2}{(k!)^2\Gamma(-s)\Gamma(2\lambda+s+k+1)\Gamma(2\lambda)^2}\times\\
%&\sum_{\ell=0}^k\frac{(-k)_\ell(k+2\lambda)_\ell(s+\lambda+1/2)_\ell(s+1)_\ell}{m!(\lambda+1/2)_\ell(2\lambda+s+k+1)_\ell(s-k+1)_\ell},
%\end{align*}
%and the lemma follows from $\Gamma(-s)=-\Gamma(1-s)/s$.
\end{proof}

\begin{lem}\label{lem:auxL3}
The following equality holds for all $\lambda\in(-1/2,0)\cup(0,\infty)$  and integer $k\geq0$:
\begin{multline*}
\int_{-1}^1C^\lambda_k(x)^2(1-x^2)^{\lambda-1/2}\log(1-x)\,dx=\frac{\pi2^{1-2\lambda}\Gamma(2\lambda+k)}{k!(k+\lambda)\Gamma(\lambda)^2}\times
\\\left(-2\psi(2\lambda+2k)+\psi(2\lambda+k)+\log2 +\psi(\lambda+k+1/2)-\frac{1}{2k+2\lambda}\right).
\end{multline*}
In other words (from the definition of $\widehat C_k^\lambda$ and $w^\lambda$),
\begin{multline*}
\int_{-1}^1\widehat C^\lambda_k(x)^2w^\lambda(x)\log(1-x)\,dx=\\
\\\left(-2\psi(2\lambda+2k)+\psi(2\lambda+k)+\log2 +\psi(\lambda+k+1/2)-\frac{1}{2k+2\lambda}\right).
\end{multline*}
\end{lem}
\begin{proof}
Let $I_{\mathrm log}$ be the integral of the lemma and let $I_s$ be the integral of Lemma \ref{lem:auxL3previo}. We first claim that
\begin{equation}\label{eq:limite}
I_{\mathrm log}=\lim_{s\to0}\frac{I_s-I_0}{s},
\end{equation}
which is an application of Lebesgue's dominated convergence theorem. Indeed, let $f_s(x)=C_k^{\lambda}(x)^2(1-x^2)^{\lambda-1/2}(1-x)^s$ be the integrand of Lemma \ref{lem:auxL3previo}, and note that for all $x\in(-1,1)$ and $k\geq1$,
\[
\frac{f_{s}(x)-f_0(x)}{s}=C_k^{\lambda}(x)^2(1-x^2)^{\lambda-1/2}\frac{(1-x)^{s}-1}{s}.
\]
It is then a simple calculus exercise to conclude that for $x\in(-1,1)$ and $s\in(0,1/2)$:
\[
\left|\frac{f_{s}(x)-f_0(x)}{s}\right|\leq
\left|C_k^{\lambda}(x)^2(1-x^2)^{\lambda-1/2}\right|\left(|x|+\left|\log\frac{1}{1-x}\right|\right)
\]
which gives us an integrable function which is an uniform bound independent of $s$. We can therefore interchange the order of the integral and the limit getting \eqref{eq:limite}.
We finally compute this limit using te expresion for $I_s$ and $I_0$ given in Lemma \ref{lem:auxL3previo}. Note that $(I_s-I_0)/s=A(s)+B(s)/s$ where
\begin{align*}
A=&\frac{2^{s+2\lambda}\Gamma(s+\lambda+1/2)\Gamma(\lambda+1/2)\Gamma(k-s)\Gamma(k+2\lambda)^2}{(k!)^2s\Gamma(-s)\Gamma(2\lambda+s+k+1)\Gamma(2\lambda)^2}\times\\
&\sum_{\ell=0}^{k-1}\frac{(-k)_\ell(k+2\lambda)_\ell(s+\lambda+1/2)_\ell(s+1)_\ell}{\ell!(\lambda+1/2)_\ell(2\lambda+s+k+1)_\ell(s-k+1)_\ell},
\\
B=&\frac{2^{s+2\lambda}\Gamma(s+\lambda+1/2)\Gamma(\lambda+1/2)\Gamma(k-s)\Gamma(k+2\lambda)^2}{(k!)^2\Gamma(-s)\Gamma(2\lambda+s+k+1)\Gamma(2\lambda)^2}\times\\
&\frac{(-k)_k(k+2\lambda)_k(s+\lambda+1/2)_k(s+1)_k}{k!(\lambda+1/2)_k(2\lambda+s+k+1)_k(s-k+1)_k}-\frac{\pi2^{1-2\lambda}\Gamma(2\lambda+k)}{k!(k+\lambda)\Gamma(\lambda)^2}.
\end{align*}
It is clear that
\begin{align*}
\lim_{s\to0}A(s)=&-\frac{2^{2\lambda}\Gamma(\lambda+1/2)^2\Gamma(k)\Gamma(k+2\lambda)^2}{(k!)^2\Gamma(2\lambda+k+1)\Gamma(2\lambda)^2}\sum_{\ell=0}^{k-1}\frac{k(k+2\lambda)}{(2\lambda+k+\ell)(k-\ell)}\\
=&-\frac{2^{2\lambda}\Gamma(\lambda+1/2)^2\Gamma(k+2\lambda)}{k!\Gamma(2\lambda)^2}\sum_{\ell=0}^{k-1}\frac{1}{(2\lambda+k+\ell)(k-\ell)}
\\
=&-\frac{\pi2^{1-2\lambda}\Gamma(k+2\lambda)}{k!\Gamma(\lambda)^2(k+\lambda)}\sum_{\ell=0}^{k-1}\left(\frac{1}{2\lambda+k+\ell}+\frac{1}{k-\ell}\right)
\\
=&-\frac{\pi2^{1-2\lambda}\Gamma(k+2\lambda)}{k!\Gamma(\lambda)^2(k+\lambda)}\left(H_k+\psi(2\lambda+2k)-\psi(2\lambda+k)\right),
\end{align*}
where we have used Legendre's duplication formula (see  \cite[6.1.18]{AS1964}) and the definition of $\psi$. On the other hand, using $\Gamma(-s)=-\Gamma(1-s)/s$ and simplifying,
\begin{align*}
B(s)=
&\frac{\pi2^{2-2\lambda+s}\Gamma(k+2\lambda)\Gamma(2k+2\lambda)\Gamma(s+\lambda+1/2+k)\Gamma(s+k+1)}{(k!)^2\Gamma(\lambda)^2\Gamma(\lambda+1/2+k)\Gamma(2\lambda+s+2k+1)\Gamma(s+1)}-\frac{\pi2^{1-2\lambda}\Gamma(2\lambda+k)}{k!(k+\lambda)\Gamma(\lambda)^2},
\end{align*}
that shows that $B(0)=0$ and gives a formula which is well defined for $s\geq0$, thus allowing us to compute the derivative directly from the formula: $\lim_{s\to0}B(s)/s=B'(0)$:
\begin{align*}
B'(0)=\frac{\pi2^{1-2\lambda}\Gamma(2\lambda+k)}{k!(k+\lambda)\Gamma(\lambda)^2}\left(\log2 +\psi(\lambda+k+1/2)+\psi(k+1)-\psi(2\lambda+2k+1)-\psi(1)\right).
\end{align*}
The lemma is now proved.
%
%\[
%I=\int_{-1}^1C^\lambda_k(x)^2(1-x)^{s+\lambda-1/2}(1+x)^{\lambda-1/2}\,dx,
%\]
%and find the value of this last integral in \cite[7.314(7)]{GR2015} \textcolor{red}{(but note there is a typo in that reference!)}:
%\begin{align*}
%I=&\frac{2^{s+2\lambda}\Gamma(s+\lambda+1/2)\Gamma(\lambda+1/2)\Gamma(k-s)\Gamma(k+2\lambda)^2}{(k!)^2\Gamma(-s)\Gamma(2\lambda+s+k+1)\Gamma(2\lambda)^2}\times\\
%&\sum_{\ell=0}^k\frac{(-k)_\ell(k+2\lambda)_\ell(s+\lambda+1/2)_\ell(s+1)_\ell}{m!(\lambda+1/2)_\ell(2\lambda+s+k+1)_\ell(s-k+1)_\ell},
%\end{align*}
%and the lemma follows from $\Gamma(-s)=-\Gamma(1-s)/s$.
%where if $s=0$ the formula is defined by continuation (note that in that case the $k$--th summand contains a division by $0$ and also we have $1/\Gamma(-0)$ in the denominator).

\end{proof}
\subsubsection{Proof of Theorem \ref{TeoL3} for $\lambda\neq0$}
We are now ready to prove our theorem. From \eqref{eq:L3comosuma} and Lemma \ref{lem:auxL3} we have
\[
L_3(\lambda,n)=-2\sum_{k=0}^n\left(-2\psi(2\lambda+2k)+\psi(2\lambda+k)+\log2 +\psi(\lambda+k+1/2)-\frac{1}{2k+2\lambda}\right).
\]
We have to see that this is equal to the expresion in Theorem \ref{TeoL3}, which we do by induction on $n$. The case $n=0$ reduces to  \cite[6.3.8]{AS1964}. Moreover, by induction hypotheses we need to check that
\[
4\psi(2\lambda+2n)-2\psi(2\lambda+n)-2\log2 -2\psi(\lambda+n+1/2)+\frac{1}{n+\lambda}=A+B,
\]
where
\begin{align*}
A=&(n+1)\left(\psi(n+\lambda+1)-\psi(\lambda+1/2)\right)-n\left(\psi(n+\lambda)-\psi(\lambda+1/2)\right)
\\
B=&(n+2\lambda-1)
\big(\psi(n+\lambda-1/2)-\psi(\lambda+1/2)-2\psi(2n+2\lambda-1)+2\psi(n+2\lambda)\big)\\
&-(n+2\lambda)
\big(\psi(n+\lambda+1/2)-\psi(\lambda+1/2)-2\psi(2n+2\lambda+1)+2\psi(n+2\lambda+1)\big).
\end{align*}
This is a simple yet tedious exercise using that $\psi(z+1)=\psi(z)+1/z$ and  \cite[6.3.8]{AS1964}.
\appendix

\section{Auxiliary results}
We have used some technical results that we include here for the reader's convenience.

\begin{lem}\label{LemaHiperFactorial}
$$\sum_{j=1}^{n}j\log(j)
=\frac{1}{2}n^2\log n -\frac{1}{4}n^2+\frac{1}{2}n\log n+\frac{1}{12}\log n+O(1)\,.$$
\end{lem}
\begin{proof}
%\textcolor{red}{Aquí pongo una demostración más elemental que la que había por no usar la función de Barnes}
%Let
%\[
%R_n=\sum_{j=1}^{n}j\log(j)
%-\frac{1}{2}n^2\log n+\frac{1}{4}n^2-\frac{1}{2}n\log n-\frac{1}{12}\log n.
%\]
%It is straightforward to compute
%\[
%R_{n+1}-R_n=-\left(\frac{n^2}{2}+\frac{n}{2}+\frac{1}{12}\right)\log\left(1+\frac1n\right)+\frac{n}{2}+\frac14.
%\]
%Moreover, from the Taylor expansion of $\log(1+x)$ at $x=0$ we have that
%\[
%\left|\log\left(1+\frac1n\right)-\left(\frac1n-\frac1{2n^2}+\frac1{3n^3}\right)\right|\leq\frac{1}{4n^4},
%\]
%that readily implies
%\[
%|R_{n+1}-R_n|\leq \frac{\frac{n^2}{2}+\frac{n}{2}+\frac{1}{12}}{4n^4}+\left|-\left(\frac{n^2}{2}+\frac{n}{2}+\frac{1}{12}\right)\left(\frac1n-\frac1{2n^2}+\frac1{3n^3}\right)+\frac{n}{2}+\frac14\right|\leq\frac{1}{n^2},
%\]
%the last by direct computation. We conclude that for $n\geq1$,
%\[
%|R_n|\leq R_1+\sum_{k=1}^n\frac{1}{k^2}\leq \frac14+\frac{\pi^2}{6}\leq 2.
%\]
Let
\begin{align*}
S_n&:=\sum_{j=1}^{n}j\log(j)=\log\left(\prod_{j=1}^{n}j^j\right)
=\log\left(\prod_{j=1}^{n}\prod_{k=1}^j j\right)
=\log\left(\prod_{k=1}^{n}\prod_{j=k}^{n}j\right)
=\log\left(\prod_{k=1}^{n}\frac{n!}{(k-1)!}\right)\\
&=n\log(n!)-\log\left(\prod_{k=1}^{n}(k-1)!\right)
=n\log\Gamma(n+1)-\log G(n+1)\,,
\end{align*}
where $G(n)=(n-2)!(n-3)!\dots 1!$ is Barnes $G$-function, also called the double gamma function. The asymptotics of $G(z)$ for $z\to+\infty$ is known (see \cite[Theorem 1]{FL2001} and note the typo in \cite[5.17.5]{NIST}):
$$\log(G(z+1))=\frac{1}{4}z^2+z\log\Gamma(z+1)-\frac{1}{2}z^2\log z-\frac{1}{2}z\log z-\frac{1}{12}\log z+O(1).$$
We thus have proved that
\begin{align*}
S_n&=n\log\Gamma(n+1)-\left(\frac{1}{4}n^2+n\log\Gamma(n+1)-\frac{1}{2}n^2\log n-\frac{1}{2}n\log n-\frac{1}{12}\log n+O(1)\right)\\
&=\frac{1}{2}n^2\log(n)-\frac{1}{4}n^2+\frac{1}{2}n\log(n)+\frac{1}{12}\log(n)+O(1).
\end{align*}
\end{proof}

\begin{lem}\label{lem:basica0}
	We have
	\[
	\int_{-\pi}^{\pi} \log\frac1{\sqrt{2-2\cos\alpha}}\,d \alpha=0.
	\]
\end{lem}
\begin{proof}
First, translate the integration interval to $[0, 2\pi]$ and after that combine the change of variables $\alpha=2\pi x$ with \cite[4.384 (3)]{GR2015}.
\end{proof}
\begin{lem}\label{lem:tecnico}
	Let $f:[-1,1]^2\to\R$ be a continuous function. Then,
	\begin{multline*}
	\int_{[-1,1]^2}\frac{f(x,y)}{\sqrt{1-x^2}\sqrt{1-y^2}}\log\frac{1}{2|x-y|}\,d(x,y)=\\\frac12\int_{-\pi}^{\pi}\,d \alpha \log\frac1{\sqrt{2-2\cos\alpha}}\int_{-\pi}^{\pi}f\left(\cos\theta,\cos(\theta+\alpha))\right)\,d\theta.
	\end{multline*}
	In particular, from Lemma \ref{lem:basica0}, if $\int_{-\pi}^{\pi}f\left(\cos\theta,\cos(\theta+\alpha))\right)\,d\theta$ is constant (i.e. if its value does not depend on $\alpha$) then the integral of the lemma is $0$.
\end{lem}
\begin{proof}
	Denote by $I$ the integral in the lemma. The change of variables $x=\cos\theta$, $y=\cos\phi$ yields
	\begin{align*}
	4I=&4\int_{[0,\pi]^2} f(\cos\theta,\cos\phi)\log\frac{1}{2|\cos\theta-\cos\phi|}\,d(\theta,\phi)\\=&\int_{[-\pi,\pi]^2} f(\cos\theta,\cos\phi)\log\frac{1}{2|\cos\theta-\cos\phi|}\,d(\theta,\phi)\\=&\int_{[-\pi,\pi]^2} f(\cos\theta,\cos\phi)\log\frac{1}{\left|e^{i\theta}-e^{i\phi}\right|}\,d(\theta,\phi)\\&+\int_{[-\pi,\pi]^2} f(\cos\theta,\cos\phi)\log\frac{1}{\left|e^{i\theta}-e^{-i\phi}\right|}\,d(\theta,\phi),
	\end{align*}
where we have used the following classical fact:
	\[
	2|\cos\theta-\cos\phi|=\left|e^{i\theta}-e^{i\phi}\right|\,\left|e^{i\theta}-e^{-i\phi}\right|
	\]
	The change of variables $\phi\to-\phi$ shows that the two last integrals above are equal and hence we have
	\[
	2I=\int_{[-\pi,\pi]^2} f(\cos\theta,\cos\phi)\log\frac{1}{\left|e^{i\theta}-e^{i\phi}\right|}\,d(\theta,\phi)
	\]
	We write this last expression as an integral in the product $[-\pi,\pi]\times S^1$, where $S^1$ is the unit circle, getting
	\begin{align*}
	2I=&\int_{-\pi}^{\pi}d\theta\int_{z\in S^1}f(\cos\theta,Re(z))\log\frac1{\left|e^{i\theta}-z\right|}\, \frac{dz}{i z}\\
	=&\int_{-\pi}^{\pi}d\theta\int_{w\in S^1}f\left(\cos\theta,Re(e^{i\theta }w)\right)\log\frac1{\left|e^{i\theta}-e^{i\theta}w\right|}\,\frac{dw}{iw}\\
	=&\int_{w\in S^1}\,\frac{d w}{i w} \log\frac1{\left|1-w\right|}\int_{-\pi}^{\pi}f\left(\cos\theta,Re(e^{i\theta }w)\right)\,d\theta.
	\end{align*}
	(we have applied the isometry $S^1\to S^1$ given by $z\to w=e^{-i\theta}z$). Parametrizing again the unit circle by $w=e^{i\alpha}$ we get to
	\begin{align*}
	2I=&\int_{\alpha\in [-\pi,\pi]}\,d \alpha \log\frac1{\sqrt{2-2\cos\alpha}}\int_{-\pi}^{\pi}f\left(\cos\theta,\cos(\theta+\alpha)\right)\,d\theta,
	\end{align*}
	proving the lemma.
\end{proof}
\begin{lem}\label{lem:integralestabla}
	For any integer $k\geq1$ we have
	\[
	\frac1\pi\int_{-\pi}^{\pi}\cos(k\alpha)\log\frac{1}{\sqrt{2-2\cos\alpha}}\,d\alpha=\frac1k,\quad 	\frac1\pi\int_{-\pi}^{\pi}\cos^2(k\alpha)\log\frac{1}{\sqrt{2-2\cos\alpha}}\,d\alpha=\frac{1}{4k}.
	\]
	Moreover, for integers $k>\ell\geq1$ we have
	\[
	\frac1\pi\int_{-\pi}^{\pi}\cos(k\alpha)\cos(\ell\alpha)\log\frac{1}{\sqrt{2-2\cos\alpha}}\,d\alpha=\frac{1}{2}\left(\frac{1}{k-\ell}+\frac{1}{k+\ell}\right).
	\]
\end{lem}
\begin{proof}
	For the first integral, we proceed as in Lemma \ref{lem:basica0}
	\[
	2\int_0^{1}\cos(2k\pi x)\log\frac{1}{2\sin (\pi x)}\stackrel{\text{\cite[4.384 (3)]{GR2015}}}{=}\frac1k.
	\]
For the second integral we consider
		\begin{align*}
	x=&\frac1\pi\int_{-\pi}^{\pi}\cos^2(k\alpha)\log\frac{1}{\sqrt{2-2\cos\alpha}}\,d\alpha,\\
	y=&\frac1\pi\int_{-\pi}^{\pi}\sin^2(k\alpha)\log\frac{1}{\sqrt{2-2\cos\alpha}}\,d\alpha,
	\end{align*}
and we note that
			\begin{align*}
	x+y=&\frac1\pi\int_{-\pi}^{\pi}\log\frac{1}{\sqrt{2-2\cos\alpha}}\,d\alpha=0,\\
	x-y=&\frac1\pi\int_{-\pi}^{\pi}\cos(2k\alpha)\log\frac{1}{\sqrt{2-2\cos\alpha}}\,d\alpha=\frac{1}{2k},
	\end{align*}
	and adding these two equalities gives the desired result. For the last claim we similarly consider
	\begin{align*}
	x=&\frac1\pi\int_{-\pi}^{\pi}\cos(k\alpha)\cos(\ell\alpha)\log\frac{1}{\sqrt{2-2\cos\alpha}}\,d\alpha,\\
	y=&\frac1\pi\int_{-\pi}^{\pi}\sin(k\alpha)\sin(\ell\alpha)\log\frac{1}{\sqrt{2-2\cos\alpha}}\,d\alpha,
	\end{align*}
	which readily gives
		\begin{align*}
	x+y=&\frac1\pi\int_{-\pi}^{\pi}\cos((k-\ell)\alpha)\log\frac{1}{\sqrt{2-2\cos\alpha}}\,d\alpha=\frac{1}{k-\ell},\\
	x-y=&\frac1\pi\int_{-\pi}^{\pi}\cos((k+\ell)\alpha)\log\frac{1}{\sqrt{2-2\cos\alpha}}\,d\alpha=\frac{1}{k+\ell},
	\end{align*}
	and again adding these equalities gives the last integral of the lemma.
\end{proof}

\begin{lem}\label{lem:harmonicsum}
	For $n\geq2$,
	\[
	\sum_{k=2}^nH_{2k-1}=nH_{2n-1}+\frac{H_{2n}}{2}-\frac{H_n}{4}-n-\frac12
	\]
\end{lem}
\begin{proof}
This is an easy exercise of induction.
\end{proof}
%\begin{lem}\label{lem:psi}
%For all $z\in\C$, $z\neq0,-1,-2,\ldots$, we have
%\[
%\Gamma(2z)=\frac{2^{2z-1}\Gamma(z)\Gamma(z+1/2)}{\sqrt{\pi}},\quad 2\psi(2z)=\psi(z)+\psi(z+1/2)+2\log2.
%\]
%\end{lem}
%\begin{proof}
%The first claim is Legendre's duplication formula. The second one is obtained by differentiating the first one w.r.t. $z$.
%\end{proof}

\begin{lem}\label{lem:trigb}
The following equality holds:
\[
\sum_{k=0}^n\cos^2\left(k\theta+\alpha\right)=\frac{n}{2}+\frac{O(1)}{\sin\theta}=\frac{n+1}{2}+\frac{O(1)}{\sin\theta}.
\]
\end{lem}
\begin{proof}
From the double angle formulas, the sum in the lemma is equal to $A-B+C$ where
\begin{align*}
A=&\cos^2(\alpha)\sum_{k=0}^n\cos^2(k\theta),\\
B=&2\cos(\alpha)\sin(\alpha)\sum_{k=0}^n\cos(k\theta)\sin(k\theta)=\frac{\sin(2\alpha)}{2}\sum_{k=0}^n\sin(2k\theta),\\
C=&\sin^2(\alpha)\sum_{k=0}^n\sin^2(k\theta).
\end{align*}
These three sums are known, see \cite[Sec. 1.34, 1.35]{GR2015} yielding the following value for the sum in the lemma:
\[
\cos^2(\alpha)+\frac{n}{2}+\cos(2\alpha)\frac{\cos((n+1)\theta)\sin(n\theta)}{2\sin\theta}-\frac{\sin(2\alpha)}{2}\frac{\sin\left((n+1)\theta\right)\sin\left(n\theta\right)}{\sin\theta}.
\]
We are done.
\end{proof}

\begin{lem}\label{lem:SumCos}
The following identities hold:
$$\displaystyle \sum_{k=0}^n\cos(ak+b)=\frac{1}{2\sin(a/2)}\left(\sin\left(an+\frac{a}{2}+b\right)-\sin\left(b-\frac{a}{2}\right)\right)\,,$$
\begin{align*}
\sum_{k=0}^n\cos(ak+b)\cos(ck+d)&=
\frac{1}{4\sin\left(\frac{a+c}{2}\right)}
\left(\sin\left((a+c)n+\frac{a+c}{2}+b+d\right)+\sin\left(\frac{a+c}{2}-b-d\right)\right)\\
&\quad+\frac{1}{4\sin\left(\frac{a-c}{2}\right)}
\left(\sin\left((a-c)n+\frac{a-c}{2}+b-d\right)+\sin\left(\frac{a-c}{2}-b+d\right)\right)
\end{align*}
\end{lem}
\begin{proof}
First identity is direct consequence of \cite[1.341.3]{GR2015} and an identity for the product of a pair of trigonometric functions.
%$$\sum_{k=0}^n\cos(ak+b)=\cos\left(b+\frac{n}{2}a\right)\sin\left(\frac{(n+1)a}{2}\right)
%\frac{1}{\sin\frac{a}{2}}
%=\frac{1}{\sin(a/2)}\frac{1}{2}\left(\sin\left(na+b+\frac{a}{2}\right)
%-\sin\left(b-\frac{a}{2}\right)\right)\,.$$
Second identity is consequence of a formula for the product of two trigonometric functions, and first identity.
%\begin{align*}
%&\sum_{k=0}^n\cos(ak+b)\cos(ck+d)=\sum_{k=0}^n\frac{1}{2}\left(\cos((a+c)k+b+d)+\cos((a-c)k+(b-d))\right)\\
%=&\frac{1}{2}\left(
%\frac{1}{2\sin\left(\frac{a+c}{2}\right)}
%\left(\sin\left((a+c)n+\frac{a+c}{2}+b+d\right)-\sin\left(b+d-\frac{a+c}{2}\right)\right)\right.\\
%&\qquad+\left.\frac{1}{2\sin\left(\frac{a-c}{2}\right)}
%\left(\sin\left((a-c)n+\frac{a-c}{2}+b-d\right)-\sin\left(b-d-\frac{a-c}{2}\right)\right)
%\right)
%\end{align*}
\end{proof}

\begin{lem}\label{integ=0} The following identity holds:
$$\int_{0}^{\pi} \log\frac{1}{2\sin v} dv=0.$$
\end{lem}
\begin{proof}
Using the change of variable $x=\cos(v)$ and taking into account that we can write $\sqrt{1-x^2}=|x-1|^{1/2} \cdot |x+1|^{1/2}$ we get
\begin{align*}
\int_{0}^{\pi} \log\frac{1}{2\sin v} dv=-\log2+\int_{-1}^{1} \log\frac{1}{\sqrt{1-x^2}} \frac{dx}{\pi\sqrt{1-x^2}}=-\log2+\frac{V(1)+V(-1)}{2},
\end{align*}
where $V(x)$ is the equilibrium measure potential, which constantly equals $\log2$.
So the integral vanishes.
\end{proof}
\begin{lem}\label{senointegAsymp}
Let $\textrm{Si}(t)$ the Integral Sine function. The following asymptotic expansion holds:
$$
\int_{0}^{x} \frac{Si(t)}{t}dt=\frac\pi2\log x + \frac{\gamma \pi}{2} + O(x^{-1}), \quad \text{ as } x\to\infty,
$$
being $\gamma$ the Euler constant.
\end{lem}
\begin{proof}
We will obtain this identity by means of the imaginary part of a complex line integral.
Let us take $x>0$ and consider $C_x = \{z\in\mathbb{C} \,:\, |z|=x; \textrm{Im } z>0\}$ parametrized counterclockwise. Then,
\begin{equation}\label{int0}
\int_{-x}^x \frac{1}{t} \int_{0}^{t} \frac{e^{iz}-1}{z} dz\,dt + \int_{C_x} \frac{1}{t} \int_{0}^{t} \frac{e^{iz}-1}{z} dz\,dt =0
\end{equation}
since we are integrating an entire function along a closed curve.
We now split the integral over the semicircle within three other ones,
\begin{align*}
\int_{C_x} \frac{1}{t} \int_{0}^{t} \frac{e^{iz}-1}{z} dz\,dt & = \int_{C_x} \frac{1}{t} \int_{0}^{x} \frac{e^{iz}-1}{z}  dz\,dt + \int_{C_x} \frac{1}{t} \int_{x}^{t} \frac{e^{iz}}{z} dz\,dt - \int_{C_x} \frac{1}{t} \int_{x}^{t} \frac{1}{z} dz\,dt
\\&=
I_1+I_2-I_3.
\end{align*}
Let us work with each one of these integrals.
\begin{equation}
I_1=(\log (-x)-\log(x)) \int_{0}^{x} \frac{e^{iz}-1}{z}  dz = i \pi \int_{0}^{x} \frac{e^{iz}-1}{z}  dz
\end{equation}
For the second integral, we parametrize $C_x$ as $t=xe^{i\theta}$ and after the change of variable $z=xe^{i\sigma}$, standard computations lead to
\begin{equation}
I_2=-\int_{0}^{\pi} \int_{0}^{\theta} \exp(ixe^{i\sigma}) d\sigma\,d\theta,
\end{equation}
which can be bounded as
\begin{equation}
|I_2|\leq\int_{0}^{\pi} \int_{0}^{\theta} e^{-x\sin\sigma} d\sigma\,d\theta = \int_{0}^{\pi} \int_{\sigma}^{\pi} e^{-x\sin\sigma} d\theta\,d\sigma \leq 2\pi  \int_{0}^{\pi/2} e^{-x\sin\sigma} d\sigma
\end{equation}
Now we use the Jordan's inequality: $\sin\sigma\geq2\sigma/\pi$ for $\sigma\in(0,\pi/2)$, getting
\begin{equation}
|I_2|\leq 2\pi  \int_{0}^{\pi/2} e^{-2x\sigma/\pi} d\sigma = \frac{\pi^2}{x}(1-e^{-x})=O(x^{-1}) \quad \textrm{ as } \quad x\to+\infty.
\end{equation}
The last integral can be computed directly using the parametrization $t=xe^{i\theta}$,
\begin{equation}
I_3=\int_{C_x}\frac{1}{t}(\log t-\log x) dt = -\frac{\pi^2}{2}.
\end{equation}
Then, \eqref{int0} reads as
\begin{equation}
\int_{-x}^x \frac{1}{t} \int_{0}^{t} \frac{e^{iz}-1}{z} dz\,dt + i \pi \int_{0}^{x} \frac{e^{iz}-1}{z}  dz + O(x^{-1}) + \frac{\pi^2}{2} =0,
\end{equation}
from where, taking imaginary part we get
\begin{equation}
2 \int_{0}^{x} \frac{\mathrm{Si}(t)}{t}dt +\pi \int_{0}^x \frac{\cos (z)-1}{z}dz + O(x^{-1}) = 0.
\end{equation}
Finally, in the second term we use the asymptotics of the integral cosine function (\cite[5.2.2, 5.2.9, 5.2.34 and 5.2.35]{AS1964}),
\begin{align*}
\mathrm{Ci}(x)=\gamma+\log x + \int_{0}^x \frac{\cos (z)-1}{z}dz = O(x^{-1})  \quad \textrm{ as } \quad x\to+\infty,
\end{align*}
and then the announced result is proved.
\end{proof}

\bibliographystyle{amsplain}
\bibliography{bibliografia}
\end{document}